\documentclass[a4paper,12pt]{article}
\usepackage{amsfonts}
\usepackage{amssymb}
\usepackage{latexsym}
\usepackage{amsmath}
\usepackage{amscd}
\usepackage{amsthm}
\usepackage{graphicx}
\usepackage{epsfig}
\usepackage{hyperref}

\usepackage{url}
\usepackage{authblk}

\newtheorem{theorem}{Theorem}[section]
\newtheorem{definition}[theorem]{Definition}

\newtheorem{proposition}[theorem]{Proposition}
\newtheorem{lemma}[theorem]{Lemma}
\newtheorem{remark}[theorem]{Remark}
\newtheorem{corollary}[theorem]{Corollary}

\newcommand{\NN}{{\mathcal{N}}}

\newcommand{\E}{{\mathbb E}}
\newcommand{\N}{{\mathbb N}}
\newcommand{\Z}{{\mathbb Z}}
\renewcommand{\P}{{\mathbb P}}
\newcommand{\Q}{{\mathbb Q}}
\newcommand{\R}{{\mathbb R}}

\newcommand{\x}{{\mathbf x}}

\newcommand{\WW}{{\mathcal W}}

\usepackage[normalem]{ulem}
\usepackage{color}

\title{Convergence of the dynamical discrete web to the dynamical Brownian web}

\author[1]{Krishnamurthi Ravishankar}
\author[2]{Kumarjit Saha}
\affil[1]{SUNY, College at New Paltz}
\affil[2]{Ashoka University}

\AtEndDocument{\bigskip{\footnotesize%
  \textsc{Department of Mathematics, 1 Hawk Drive, New Paltz, NY 12561} \par  
  \textit{E-mail address}, K.~Ravishankar: \texttt{ravi101048@gmail.com} \par
  \addvspace{\medskipamount}
  \textsc{Department of Mathematics, Ashoka University, Sonipat, India} \par
  \textit{E-mail address},  K.~Saha: \texttt{kumarjitsaha@gmail.com}  \par
}}

\begin{document}
\maketitle

\begin{abstract}
In this paper we study the convergence of  dynamical discrete web (DyDW)  
to the dynamical Brownian web (DyBW) in the path space topology. 
We show that  almost surely the DyBW has RCLL paths taking values 
in an appropriate metric space and as a sequence of RCLL paths, 
the scaled dynamical discrete web  converges to the DyBW. This proves 
weak convergence of the DyDW process to the DyBW process. 
\end{abstract}

\section{Introduction and main result}

In this paper, we present a number of results concerning a dynamical version 
of the Brownian web (BW), known as the dynamical Brownian web (DyBW) and convergence to
it for a dynamical version of coalescing random walks model known as the 
dynamical discrete web (DyDW). The DyDW was introduced by Howitt and Warren \cite{HW09}, is a system
of coalescing simple symmetric one-dimensional random walks evolving over 
a dynamic time interval $[0,1]$. We first describe the DyDW model briefly.  

The discrete web (DW) is a collection of one-dimensional simple symmetric 
random walks starting from every point in the discrete space-time domain 
$\Z^2_{\text{even}} := \{(x,t) \in \Z^2 : x + t \text{ even}\}$.    
We consider an i.i.d. collection of random variables 
$\{I_{(x,t)} : (x,t) \in \Z^2_{\text{even}}\}$
such that 
\begin{align*}
I_{(0,0)} := 
\begin{cases}
+ 1 & \text{ w.p. } 1/2\\
- 1 & \text{ w.p. } 1/2,
\end{cases}
\end{align*}   
where $I_{(x,t)}$ gives the random walk increment at location $x$ at time $t$.
Following the walk, the walker at $(x,t)$ goes to $(x + I_{(x,t)}, t+1)\in \Z^2_{\text{even}}$ along the edge $\langle (x,t),(x + I_{(x,t)}, t+1)\rangle$. Thus, following 
the edges in DW we get a (random) collection of paths.    
The motivation for calling this system as the DW comes from the fact that 
under diffusive scaling as a collection 
of paths, the DW converges to the BW (Theorem 6.1 of \cite{FINR04}).  
For the dynamical version, we start with an ordinary DW at dynamic time $0$
and at each lattice point, direction of the outgoing edge switches at a fixed rate 
independently of all other lattice points. This gives DyDW as a process of collection of 
paths evolving dynamically over the time interval $[0,1]$.     

Howitt and Warren \cite{HW09} rightly guessed that if we slow down the rate of switching
(of outgoings edges) by $1/\sqrt{n}$, then under diffusive scaling,
there should be a non-trivial scaling limit the dynamical Brownian web
process $\{\WW(s) : s \in [0,1]\}$.
Assuming existence of such a process, it's finite dimensional 
distributions were analysed in \cite{HW09} and it was shown that 
for $s_1, s_2 \in [0,1]$, the distribution of $(\WW(s_1), \WW(s_2))$ is given by 
{\it sticky pair} of Brownian webs with degree of stickiness is given by $1/|2(s_2 - s_1)|$
(for a definition of sticky pair of Brownian webs see Section \ref{subsec:Sticky_BW}). Existence of a consistent family of finite dimensional distribution for $\WW(\tau)$ follows from this and stationarity and Markov property proved in Theorem 6.2 of \cite{NRS10} .   In \cite{NRS10}, Newman et al. proved that such a process  uniquely exists 
  and provided a rigorous construction of the DyBW as well. 
In another work \cite{NRS09A}, a sketch of the proof of convergence of 
finite dimensional distributions of
  DyDW  to that of DyBW was given.  
But to the best of our knowledge, weak convergence of the DyDW process 
 to the DyBW process has not been shown so far. 
Our goal in this paper is to provide a stronger topological setting for studying 
this convergence and give a proof for convergence in that setting,
namely as a process with RCLL paths. We state our result in detail in
Theorem \ref{thm:DYDW_Conv}. Towards this, we established that 
the DyBW process has RCLL paths  a.s. (taking values in an appropriate metric space). 
We prove this in Theorem \ref{thm:DYBW_RCLL}.       

The paper is organised as follows. In Section \ref{sec:DyBW_RCLL}
we prove that the DyBW has RCLL paths taking values in an appropriate metric space a.s.
Details of the relevant metric space have also been described.  
 In Section \ref{sec:FDD_DyDW} we describe 
the DyDW model and prove that it's finite dimensional distributions converges to that of 
DyBW.  The main argument for the same was already developed in \cite{NRS09A}. 
Finally, in Section \ref{sec:Tight} we prove that as a sequence of RCLL paths, 
the sequence of scaled dynamic discrete webs is tight and hence, we have process level convergence.

\section{DyBW is a.s. RCLL path process}
\label{sec:DyBW_RCLL}

In this section we show that the DyBW process has RCLL paths a.s. (Theorem \ref{thm:DYBW_RCLL}). 
The standard BW originated in
the work of Arratia ( see \cite{A79} and \cite{A81})
as the scaling limit of the voter model on $\Z$. Later T\'{o}th and Werner \cite{TW98} 
gave a construction of a system of coalescing Brownian motions starting from every point in space-time plane $\R^2$ and used it to construct the true self-repelling motion.
Intuitively, the BW can be thought of as a collection
of one-dimensional coalescing Brownian motions starting from every point in the space time
plane $\R^2$. Later  Fontes \textit{et. al.} \cite{FINR04} provided a framework in
which the Brownian web is realized as a random variable taking values in a Polish space, 
which enabled them to prove the weak convergence to the Brownian web  of coalescing random walks starting from every point of the space-time lattice in the diffusive scaling limit.
In the following section we recall relevant details from \cite{FINR04} to describe the 
Polish space of our interest.

\subsection{Preliminary details about Polish space of collection of paths}
\label{subsec:Preliminary}

Let $\R^{2}_c$ denote the completion of the space time plane $\R^2$ with
respect to the metric
\begin{equation*}
\rho((x_1,t_1),(x_2,t_2))=|\tanh(t_1)-\tanh(t_2)|\vee \Bigl| \frac{\tanh(x_1)}{1+|t_1|}
-\frac{\tanh(x_2)}{1+|t_2|} \Bigr|.
\end{equation*}
As a topological space $\R^{2}_c$ can be identified with the
continuous image of $[-\infty,\infty]^2$ under a map that identifies the line
$[-\infty,\infty]\times\{\infty\}$ with the point $(\ast,\infty)$, and the line
$[-\infty,\infty]\times\{-\infty\}$ with the point $(\ast,-\infty)$.
A path $\pi$ in $\R^{2}_c$ with starting time $\sigma_{\pi}\in [-\infty,\infty]$
is a mapping $\pi :[\sigma_{\pi},\infty]\rightarrow [-\infty,\infty]$ such that
$\pi(\infty)= \pi(-\infty)= \ast$ and $t\rightarrow (\pi(t),t)$ is a continuous
map from $[\sigma_{\pi},\infty]$ to $(\R^{2}_c,\rho)$.
We then define $\Pi$ to be the space of all paths in $\R^{2}_c$ with all possible starting times in $[-\infty,\infty]$.
The following metric, for $\pi_1,\pi_2\in \Pi$
\begin{equation*}
d_{\Pi} (\pi_1,\pi_2)= |\tanh(\sigma_{\pi_1})-\tanh(\sigma_{\pi_2})|\vee\sup_{t\geq
\sigma_{\pi_1}\wedge
\sigma_{\pi_2}} \Bigl|\frac{\tanh(\pi_1(t\vee\sigma_{\pi_1}))}{1+|t|}-\frac{
\tanh(\pi_2(t\vee\sigma_{\pi_2}))}{1+|t|}\Bigr|
\end{equation*}
makes $\Pi$ a complete, separable metric space. Convergence in this
metric can be described as locally uniform convergence of paths as
well as convergence of starting times.
Let ${\cal H}$ be the space of compact subsets of $(\Pi,d_{\Pi})$ equipped with
the Hausdorff metric $d_{{\cal H}}$ given by,
\begin{equation*}
d_{{\cal H}}(K_1,K_2)= \sup_{\pi_1 \in K_1} \inf_{\pi_2 \in
K_2}d_{ \Pi} (\pi_1,\pi_2)\vee
\sup_{\pi_2 \in K_2} \inf_{\pi_1 \in K_1} d_{\Pi} (\pi_1,\pi_2).
\end{equation*}
The space $({\cal H},d_{{\cal H}})$ is a complete separable metric space. Let
$B_{{\cal H}}$ be the Borel  $\sigma-$algebra on the metric space $({\cal H},d_{{\cal H}})$.
The Brownian web ${\cal W}$ is an $({\cal H},B_{{\cal H}})$ valued random
variable. Below we state Theorem 2.1 of \cite{FINR04} characterizing the Brownian web 
as a ${\cal H}$-valued random variable. 
\begin{theorem}[Theorem 2.1 of \cite{FINR04}]
\label{theorem:Bwebcharacterisation}
There exists an $({\mathcal H}, {\mathcal B}_{{\mathcal H}})$-valued random variable
${\mathcal W}$ whose distribution is uniquely determined by
the following properties:
\begin{itemize}
\item[$(a)$] from any deterministic point $\x\in\R^2$, there is  almost surely a unique path $\pi^{\x}\in {\mathcal W}$  starting from $\x$;
\item[$(b)$] for a finite set of deterministic points $\x^1,\dotsc, \x^k \in \R^2$, the collection $(\pi^{\x^1},\dotsc,\pi^{\x^k})$ is distributed as coalescing Brownian motions starting from $\x^1,\dotsc,\x^k$;
\item[$(c)$] for any countable deterministic dense set ${\mathcal D}$ of $\R^2$, ${\mathcal W}$ is the closure of $\{\pi^{\x}: \x\in {\mathcal D} \}$ in $(\Pi, d_{\Pi})$  almost surely.
\end{itemize}
\end{theorem}
The DyBW $\{\WW(\tau) : \tau \in [0,1]\}$
is defined as a ${\cal H}$-valued process evolving over dynamic time domain $[0,1]$
with finite dimensional distributions as mentioned in Section \ref{subsec:Sticky_BW}.
Newman et. al. \cite{NRS10} provided a rigorous construction and showed that such a process 
indeed exists. Our next result shows that the DyBW process has ${\cal H}$-valued RCLL paths
a.s.
\begin{theorem}
\label{thm:DYBW_RCLL}
The DyBW process $\{ \WW(\tau) : \tau \in [0,1]\}$ has RCLL paths a.s.
\end{theorem}
In what follows, we denote the Polish space of ${\cal H}$-valued RCLL paths defined over 
$[0,1]$ with Skorohod metric as ${\cal D}^{{\cal H}}([0,1])$. In other words, 
Theorem \ref{thm:DYBW_RCLL} gives us that the DyBW is a 
${\cal D}^{{\cal H}}([0,1])$ valued random variable.

While providing a a rigorous construction of 
the DyBW $\{\WW(\tau) : \tau \in [0,1]\}$, Newman et. al. developed
 an alternate construction of the Brownian net $\NN$ as well. Their work presents a  
 construction of the DyBW and the Brownian net both constructed on a {\it common} probability space. 
 In this work we will extensively use this correspondence 
 and refer this as `the {\it corresponding} net $\NN$' of the 
 DyBW $\{\WW(\tau) : \tau \in [0,1]\}$ and vice versa.
In the next section we describe their common construction and the corresponding net briefly. 
For details we refer the reader to Theorem 5.5 and Proposition 6.1 of \cite{NRS10}.


\subsection{The `corresponding' Brownian net}
\label{subsec:Corr_Net}

The approach of Newman et. al. \cite{NRS10} is based on the 
construction of a certain Poissonian marking 
which is governed by local time of the forward web along
the backward web, i.e., construction of a three-dimensional Poisson point 
process with intensity measure $ L \times \ell$,
where $\ell$ is Lebesgue measure and $L$ is the local time measure of the
forward web along the backward web. 
For a deterministic point $(x,t) \in \R^2$ almost surely 
there is exactly one outgoing path starting at $(x,t)$ and no incoming path passing 
through $(x,t)$ in the Brownian web. 

But there are random $(1, 2)$  points in the BW where a single Brownian
web path enters the point from an earlier time and two paths leave from that. 
Among the two outgoing paths, exactly one path is 
the continuation of the (incoming) path coming from earlier time 
and the other one is ``newly-born" (see Figure \ref{fig:OneTwo}).

\begin{figure}[htb!]
\label{fig:OneTwo}
\begin{center}
\includegraphics[scale=.9]{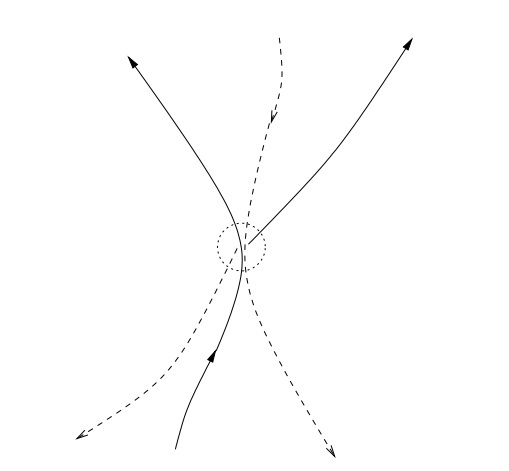}
\caption{ A schematic diagram of a $(1,2)_l$ point as the incoming
forward path connects to the leftmost outgoing path. The other outgoing path represents
the ``newly born path". The dotted lines represent similar 
behaviour for the backward paths. It shows that at a $(1,2)$ point,  a backward path
meets a forward path.}
\end{center}
\end{figure}

Further, $(1, 2)$ points of BW are precisely those at which a forward and
a backward path meet and the set of `marked
points' will be supported on the set of $(1,2)$ points.
Each $(1, 2)$ point has a preferred left or right ``direction" and 
accordingly it is denoted as $(1, 2)_l$ or $(1, 2)_r$.
For a $(1, 2)_r$ ($(1, 2)_l$) point, the continuing path 
(coming in from earlier time) is to the right (left) of the ``newly-born" path. 
 One of the major observation of \cite{NRS10} is that at the continuum level, 
the analogue of DyDW switching will simply be change of directions at all marked (1, 2) points. 
In other words, the DyBW  at dynamic time $s_0$, denoted by $\WW(s_0)$, will
be deduced from the initial BW $\WW(0)$ by switching directions
 of all marked $(1, 2)$ points, marked during dynamic time interval $[0, s_0]$. 
 We should mention that the situation is  
much more complicated here since, $L$ is not 
locally finite as the set of $(1,2)$ points is dense in $\R^2$. 
However, as showed in \cite{NRS10}, one can approximate $L$ by a sequence of locally finite measures $L_n$
and do the markings using $L_n$, and then let $n \to \infty$. 

Further, it was shown that the same marked point process gives an alternate 
construction of the Brownian net, viz., for a marked $(1, 2)$ point,
which was originally an $(1,2)_l$ point say, the Brownian net includes not only 
paths that connect to the left outgoing path (as in the original web) but also ones
that connect to the right outgoing path. Formally, 
 for deterministic $(x,t) \in \R^2$ consider the a.s. unique $\pi^{(x,t)}$ starting 
 from $(x,t)$ in the initial BW $\WW(0)$. For each $(x,t) \in \R^2$
 consider the set of all the paths obtained by continuing along both the  
 right and left outgoing paths at each `marked' $(1,2)$ point on the 
 trajectory  of $\pi^{(x,t)}$ and this collection gives the standard Brownian net (Theorem 5.5
 of \cite{NRS10}). The above construction using Poissonian marking of (1,2) points allows
 us to have the standard DyBW and the standard Brownian net defined on the same probability space.
 In this paper, we refer this as the ``corresponding" Brownian net $\NN$ corresponding to the 
 DyBW $\{\WW(s) : s \in [0,1]\}$ and we will use this correspondence extensively.  
 We note that the Brownian net uniquely defines the dual net. This allows us to 
define the  ``the corresponding" double Brownian net vector $(\NN, \widehat{\NN})$ where   
$\widehat{\NN}$ is the dual for $\NN$.   

We note here that for any interval (open or closed) $A \subseteq [0 , 1]$
we use the same construction to construct the `corresponding' net $\NN_{A}$ 
corresponding to the DyBW process $\{\WW(\tau) : \tau \in A \}$
considering `markings' only in the dynamic time interval $A$. 
For any deterministic $0 \leq \tau_1 < \tau_2 \leq 1$ we have 
$$
\NN_{[\tau_1, \tau_2]} = \NN_{[\tau_1, \tau_2)} = \NN_{(\tau_1, \tau_2)} = 
 \NN_{(\tau_1, \tau_2]}\text{ a.s. }
$$
Before ending this section we need to define some important quantities 
associated with the corresponding Brownian net $\NN$.   

For any path $\pi \in \Pi$ and for $t \geq \sigma_\pi$, 
its restriction over the time domain $[t, \infty)$ is denoted as $\pi_{[t, \infty)}$.   
 For $\pi_1, \pi_2 \in \Pi$ and 
for any interval $A \subset [\sigma_{\pi_1}\vee \sigma_{\pi_2}, \infty)$,
the notation $d_{\Pi} (\pi_1,\pi_2)_{A}$ is defined as 
\begin{equation*}
d_{\Pi} (\pi_1,\pi_2)_{A}
= \sup_{s \in A} \Bigl|\frac{\tanh(\pi_1(s))}{1+|s|}-\frac{
\tanh(\pi_2(s))}{1+|s|}\Bigr|.
\end{equation*}
In other words, the quantity $d_{\Pi} (\pi_1,\pi_2)_{A}$ does not consider  
distance between starting points and represents  
distance between restricted trajectories of these two paths  w.r.t. 
spatial part of $d_\Pi$ metric restricted over the time interval $A \subset [\sigma_{\pi_1}\vee \sigma_{\pi_2}, \infty)$.
For a general collection of paths we describe the notion of 
separation points as follows.

\begin{definition}
\label{def:SeparationPt}
For a path family ${\cal K} \subset \Pi$, a point $(x,t) \in \R^2$
is said to be a `separation point' if there exist $\pi_1, \pi_2 \in {\cal K} $
such that the following conditions are satisfied :
\begin{itemize}
\item[(i)] $\sigma_{\pi_1} \vee \sigma_{\pi_2} < t$ and $\pi_1(t) = \pi_2(t) = x$, i.e., 
both the paths start strictly before time $t$ and pass through the point $(x,t)$;
\item[(ii)] The first meeting time of $\pi_1, \pi_2$ after $t$ defined as 
$$
t^c_{\pi_1, \pi_2} := \inf\{ s > t : \pi_1(s) = \pi_2(s) \},
$$
is strictly bigger than $t$.
\end{itemize}
\end{definition}
Let ${\cal S} = {\cal S}({\cal K})$ denote the set of separation 
points in ${\cal K}$. Definition of a separating point suggests that for each 
separating point $ (x,t) \in {\cal S}({\cal K})$ there exist paths $\pi_1, \pi_2 \in {\cal K}$
which get separated at $(x,t)$ and create an excursion set, which 
is a subset of $\R^2$, in between them over the time interval 
$(t, t^c_{\pi_1, \pi_2})$. The (random) set of separation points in the corresponding 
Brownian net is denoted as ${\cal S}(\NN)$. We mention here that 
our notion of separation points for $\NN$ is slightly different from the notion of $\NN$
separation points as introduced in \cite{NRS10} which we mention below:
\begin{definition}[Definition 7.2 \cite{NRS10}]
\label{def:NRS_separation}
For $ - \infty <T_1 < T_2 < +\infty$ a (random) point $(x, t)$ 
with $T_1 < t < T_2$ is said to be a $(T_1,T_2)$
 separation point iff there are two paths $\pi_1$ and $\pi_2$ in the 
 Brownian net $\NN$ starting from $\R \times \{T_1\}$ and separating at $(x, t)$
  which do not touch (intersect) on $(t, T_2]$.
\end{definition}
Let $S^{T_2}_{T_1}$ denote the set of all $(T_1, T_2)$ separating points 
in the corresponding Brownian net $\NN$. We observe that $S^{T_2}_{T_1} \subseteq {\cal S}(\NN)$ a.s.
for all $T_1 < T_2 $.

Clearly, for Brownian net a separation point must necessarily 
be a marked $(1,2)$ point. There can be several excursion sets created at a separation point 
$(x,t)$.  We are going to define 
the maximal excursion set and several other quantities related to maximal 
excursion set generated at a separation point, i.e., at a marked $(1,2)$ type 
point here: 
\begin{itemize}
\item[(1)] {\bf Maximal excursion set}  generated at a separating point
 $(x,t) \in {\cal S}({\cal K})$  is denoted by $E^{(x,t)}$ and defined as 
\begin{align*}
E^{(x,t)} := \{ & (y,s) : s \in [t, t^c_{\pi_1, \pi_2}] \text{ and } y \in [\pi_1(s), \pi_2(s)]
\text{ for some }\pi_1, \pi_2 \in {\cal K} \\
& \text{ with } \pi_1(t) = \pi_2(t) = x  \text{ and } t^c_{\pi_1, \pi_2} > t \}.
\end{align*}
\item[(2)] {\bf Diameter/width} at a separating point $(x,t) \in {\cal S}({\cal K})$ 
(of the maximal excursion set at $(x,t)$) is denoted by $D^{(x,t)}$ and defined as 
\begin{align*}
D^{(x,t)} := \sup\{ & d_\Pi(\pi_1, \pi_2)_{[t, t^c_{\pi_1, \pi_2}]} : 
\pi_1, \pi_2 \in {\cal K}, \sigma_{\pi_1} \vee \sigma_{\pi_2} < t, \\
& \text{ and }\pi_1(t) = \pi_2(t) = x \text{ with } t^c_{\pi_1, \pi_2} > t \}.
\end{align*}

\item[(3)] {\bf Survival time} of a separating point $(x,t) \in {\cal S}({\cal K})$ 
(of the maximal excursion set at $(x,t)$) is denoted by $T^{(x,t)}$ and defined as 
\begin{align*}
T^{(x,t)} := \sup\{ & \tanh(t^c_{\pi_1,\pi_2}) - \tanh(t) : 
\pi_1, \pi_2 \in {\cal K}, \sigma_{\pi_1} \vee \sigma_{\pi_2} < t, \\
& \text{ and }\pi_1(t) = \pi_2(t) = x \text{ with } t^c_{\pi_1, \pi_2} > t \}.
\end{align*}

\item[(4)] {\bf Age} of a separating point $(x,t) \in {\cal S}({\cal K})$ 
(of the maximal excursion set at $(x,t)$) is denoted by $A^{(x,t)}$ and defined as 
\begin{align*}
A^{(x,t)} := \sup\{ & \tanh(t) - \tanh(\sigma_\pi) : 
\pi  \in {\cal K}, \sigma_{\pi}  < t \text{ and }\pi(t) = x  \}.
\end{align*}
\end{itemize}

For $\epsilon > 0$, the set of separating points  in ${\cal K}$ with diameter 
more than $\epsilon$ is denoted as $D^\epsilon$ and defined as  
$$
D^\epsilon = D^\epsilon({\cal K}) := \{(x,t) \in {\cal S}({\cal K}) : D^{(x,t)} > \epsilon \}.
$$
Similarly, the sets
$A^\epsilon = A^\epsilon({\cal K}) := \{(x,t) \in {\cal S}({\cal K}) : A^{(x,t)} > \epsilon \}$ 
and $T^\epsilon = T^\epsilon({\cal K}) := \{ (x,t) \in {\cal S}({\cal K}) : T^{(x,t)} > \epsilon \}$
are defined.

 For ease of notation, for an interval $A \subseteq [0,1]$ for ease of notation, 
 let ${\cal S}_{A}$ denote the set of separation points  ${\cal S}(\NN_A)$.
Next, for a separating point $(x,t) \in {\cal S}_{A}$ based on the path family $\NN_A$, we define 
the diameter, age and survival time quantities denoted by $D^{(x,t)}_{A}, 
A^{(x,t)}_{A}, T^{(x,t)}_{A}$ respectively. Based on the above quantities, 
for $\epsilon > 0$ we define the following random set of separation points:
\begin{align*}
D^\epsilon_A & := D^\epsilon(\NN_A) \\
T^\epsilon_A & := T^\epsilon(\NN_A)\text{ and }\\
A^\epsilon_A & := A^\epsilon(\NN_A) .
\end{align*}
For deterministic $0 \leq \tau_1 < \tau_2 \leq 1$ we have 
\begin{align*}
(D^\epsilon_{[\tau_1, \tau_2]}, T^\epsilon_{[\tau_1, \tau_2]}, A^\epsilon_{[\tau_1, \tau_2]}) 
= 
(D^\epsilon_{[\tau_1, \tau_2)}, T^\epsilon_{[\tau_1, \tau_2)}, A^\epsilon_{[\tau_1, \tau_2)})
\text{ a.s.}
\end{align*}
Fix $m \in \N$. For $0 \leq j \leq m - 1$ for ease of notation we set
\begin{align*}
D^\epsilon & := D^\epsilon_{[0, 1]} \text{ and }  D^\epsilon_{j} := D^\epsilon_{[j/m, (j+1)/m]}
= D^\epsilon_{[j/m, (j+1)/m)} \\
T^\epsilon & := T^\epsilon_{[0,1]}\text{ and } T^\epsilon_{j} := T^\epsilon_{[j/m, (j+1)/m]} = 
T^\epsilon_{[j/m, (j+1)/m)}\\
A^\epsilon & := A^\epsilon_{[0, 1]} \text{ and } A^\epsilon_{j} := A^\epsilon_{[j/m, (j+1)/m]} = 
A^\epsilon_{[j/m, (j+1)/m]}.
\end{align*}

In the next section we use the correspondence 
between the DyBW and the Brownian net and prove Theorem \ref{thm:DYBW_RCLL}.

\subsection{Proof of Theorem \ref{thm:DYBW_RCLL}}
We first start with some basics on RCLL paths. Let  $(X, d_X)$ denote a general metric space. For a $(X, d_X)$ valued 
function $f$ defined over $[0,1]$ and for any subset  $ A \subset [0,1]$ we define
\begin{align*}
\omega'_f(A) := \sup\{d_X(f(s_1) , f(s_2)) : s_1 , s_2 \in A\}.
\end{align*} 
For $\delta > 0$, the notation 
$\omega_f(\delta)$ denotes the natural extension 
of  modulus of continuity for a general function $f$ defined as  
\begin{align*}
\omega_f(\delta) := \inf \sup\{\omega'_f([t_i, t_{i+1})) :  0 \leq i \leq n-1\},
\end{align*} 
where the infimum is taken over all partitions $0 = t_0 < t_1 < \cdots < t_n = 1$ of $[0,1]$,
with $t_{i+1}-t_i > \delta $ for all $0 \leq i \leq n-1$.
We have the following characterisation of RCLL functions (see page 123 of \cite{B99}):
\begin{lemma}
\label{lem:RCLL_Char}
$f$ is RCLL if and only if  $\lim_{\delta \downarrow 0} \omega_f(\delta) = 0$.
\end{lemma}
Before we proceed further, we want to make an useful observation regarding 
the function $\omega(\cdot)$ that just follows from the triangle inequality.
\begin{remark}
\label{rem:Subadditivity}
For an RCLL function $f$ and for any $0 \leq \tau_1 < \tau^\prime < \tau_2 \leq 1$ 
which is not a jump point,  we have
\begin{equation*}
\omega_f([\tau_1, \tau_2)) \leq 
\omega_f([\tau_1, \tau^\prime))  + \omega_f([\tau^\prime, \tau_2)).
\end{equation*}
\end{remark}

Considering the dynamical Brownian web (DyBW) $\{\WW(\tau) : \tau \in [0,1]\}$
as a ${\cal H}$-valued stochastic process over the dynamic time domain $[0,1]$,
with a slight of abuse of notation, we denote corresponding quantities for 
the DyBW as $\omega_{\WW}(\delta)$. Because of Lemma \ref{lem:RCLL_Char}, 
in order to prove Theorem \ref{thm:DYBW_RCLL}, it suffices to prove the following proposition.  
\begin{proposition}
\label{prop:DyBWModCont}
For any $\epsilon > 0$ a.s. there exists (random) $\delta_0 = \delta_0(\epsilon, \omega) > 0$ 
such that for all $\delta \leq \delta_0$, we have 
$\omega_{\WW}(\delta) < \epsilon$.
\end{proposition}
In other words Proposition \ref{prop:DyBWModCont} tells us that
$\lim_{\delta \downarrow 0} \omega_{\WW}(\delta) = 0$  a.s.
Proposition \ref{prop:DyBWModCont} will be proved through a sequence of lemmas.

We recall the notations introduced in the previous section  and 
it is straightforward to observe that for any $0 \leq \tau_1 < \tau_2 \leq 1$
  we have 
 $$
 D^\epsilon_{[\tau_1, \tau_2]}  \subset D^\epsilon_{[0,1]}
 \text{ and }
 T^\epsilon_{[\tau_1, \tau_2]} \subset T^\epsilon_{[0,1]} \text{ a.s.}
 $$  
Before we proceed further, we make some important remarks. 
\begin{remark}
\label{rem:EtaSet}
\begin{itemize}
\item[(i)] We observe that a $(1,2)$ type point $(x,t)$ of $\WW(\tau)$ belongs to 
the set of separating points ${\cal S}_{[\tau_1, \tau_2)}$ of the corresponding Brownian net
only if it undergoes switching  during the interval $(\tau_1, \tau_2)$
which means that the Poisson clock triggering switching 
event associated to $(x,t)$ must ring during this interval. That is it is a ``marked" point
and marking occurs during the interval $(\tau_1, \tau_2)$.
\item[(ii)] For $0 \leq \tau_1 < \tau_2 \leq 1$, let $\pi_{\tau_1}$ 
and  $\pi_{\tau_2}$ denote the trajectory of a `skeletal' 
Brownian path $\pi$ at dynamic times $\tau_1, \tau_2$ respectively. As $\pi$ is a skeletal Brownian path (starting from a point in $\Q^2$),  we 
have $\sigma_{\pi_{\tau_1}} = \sigma_{\pi_{\tau_2}} \in \Q$. 
If none of the $(1,2)$ type points on the trajectory of $\pi_{\tau_1}$ belongs
to the set $D^\epsilon_{[\tau_1, \tau_2]}$, then we must have 
$ d_\Pi(\pi_{\tau_1}, \pi_{\tau_2}) < \epsilon $.
\end{itemize}
\end{remark}
The next lemma shows that for any $(x,t) \in D^\epsilon_{[\tau_1, \tau_2]}$, the 
quantity $T^{(x,t)}_{[\tau_1, \tau_2]}$ must be sufficiently large as well. 
\begin{lemma}
\label{lem:PathAge}
Fix $\epsilon > 0$. For any interval $A \subseteq [0, 1]$ 
there exists $0 < \tilde{\epsilon} = \tilde{\epsilon}(\epsilon, \omega) < \epsilon$ 
(which does not depend on $A$) such that a.s. we have
 $ D^\epsilon_{A} \subseteq T^{\tilde{\epsilon}}_{A} $.
 \end{lemma}
\begin{proof}
We observe that 
$$
\pi \in \WW(\tau )\text{ for some }\tau \in [0,1] \text{ only if }
\pi \in \NN,
$$
where $\NN$ denotes the {\it corresponding} Brownian net. 
We also note that a.s. paths in the Brownian net $\NN$ form 
an equicontinuous path family. 
This allows us to define $ \tilde{\epsilon} = \tilde{\epsilon}(\epsilon, \omega)> 0$  as 
\begin{align}
\label{eq:epsilon^tilde_choice}
\tilde{\epsilon} := \sup\{ & \delta \in (0, \epsilon/3) : \text{ for any }\pi \in \NN \text{ and }
u_1, u_2 \geq  \sigma_{\pi}\text{ with }|\tanh(u_1) - \tanh(u_2)| \leq \delta \nonumber \\ 
&  \text{ we have } |\frac{\pi(u_1)}{1 + |u_1|} - \frac{\pi(u_2)}{1 + |u_2|}| < \epsilon/3 \}.
\end{align}
By definition, we have $\tilde{\epsilon} < \epsilon/3$ a.s. 
Equation (\ref{eq:epsilon^tilde_choice}) ensures that for any $\pi_1, \pi_2 \in \NN, t \geq 
\sigma_{\pi_1}\vee \sigma_{\pi_2}$ with  $\pi_1(t) = \pi_2(t)$
and for any interval $I \subset [t, \infty)$
satisfying $|\tanh(u) - \tanh(t)| \leq \tilde{\epsilon} $ for all $u \in I$, we have 
$d_\Pi(\pi_1, \pi_2)_I < \epsilon$. This follows from the observation that
\begin{align*}
|\frac{\pi_1(u_1)}{1 + |u_1|} - \frac{\pi_2(u_2)}{1 + |u_2|}| \leq 
|\frac{\pi_1(u_1)}{1 + |u_1|} - \frac{\pi_1(t)}{1 + |t|} | + 
| \frac{\pi_2(t)}{1 + |t|} - \frac{\pi_2(u_2)}{1 + |u_2|}| < \epsilon
\text{ for all }u_1, u_2 \in I.
\end{align*}
For any $(x,t) \in D^\epsilon_{A}$ there exist 
paths $\pi_1, \pi_2$ in the corresponding Brownian net 
$\NN_{A}$ passing through $(x,t)$  with $d_\Pi(\pi_1, \pi_2)_{[t, \tanh(t^c_{\pi_1, \pi_2})]} > \epsilon$. 
In order to have $d_\Pi(\pi_1, \pi_2)_{[t, \tanh(t^c_{\pi_1, \pi_2})]} > \epsilon$, 
we must have 
$$
\tanh(t^c_{\pi_1, \pi_2}) > \tanh(t) + \tilde{\epsilon}.
$$
 In other words, the paths $\pi_1$ and  $\pi_2$ are allowed to intersect only after time $t_0$
where $\tanh(t_0) - \tanh(t) > \tilde{\epsilon}$.
This implies that $T^{(x,t)}_{[\tau_1, \tau_2]} > \tilde{\epsilon}$
 and completes the proof.  
\end{proof}

Next, for $\epsilon > 0$ and for any interval $A \subset [0 , 1]$ we define 
\begin{align}
\Xi^\epsilon_{A} = \Xi^\epsilon(\NN_{A}) := 
D^\epsilon_{A} \cap \text{A}^{\tilde{\epsilon}}_{A},
\end{align}
where $\tilde{\epsilon} > 0$ is as in (\ref{eq:epsilon^tilde_choice}).
For simplicity of notations, we set $\Xi^\epsilon := \Xi^\epsilon_{[0, 1]}$.
It is not difficult to observe that for any $0 \leq \tau_1 < \tau_2 \leq 1$ we have
$ \Xi^\epsilon_{[\tau_1, \tau_2]} \subset \Xi^\epsilon_{[0, 1]}$.  

Lemma \ref{lem:PathAge} gives us that  
$$
\Xi^\epsilon_{[\tau_1, \tau_2]} \subseteq 
\text{A}^{\tilde{\epsilon}}_{[\tau_1, \tau_2]} 
\cap T^{{\tilde{\epsilon}}}_{[\tau_1, \tau_2]}.
$$ 
\begin{lemma}
\label{lem:CmptBox}
Fix $\epsilon > 0$. There exists $M = M(\epsilon, \omega) \in \N $ such that
for any interval $A \subset [0,1]$ we have 
\begin{align}
\label{eq:CmptBox}
\Xi^\epsilon_{A} \subseteq \Xi^\epsilon_{[0,1]} \subset [-M, M]^2. 
\end{align}
\end{lemma} 

\begin{proof}
By definition, for any $(x,t) \in \Xi^\epsilon$ we have 
$$
A^{(x,t)}_{[0,1]}\wedge T^{(x,t)}_{[0,1]} > \tilde{\epsilon} > 0.
$$
Hence, we can choose  $M_1 > 0$ such that for any $(x,t) \in \Xi^\epsilon$ 
both the following conditions hold:

\begin{itemize}
\item[(i)] $t \in [-M_1, M_1]$ as well as  
\item[(ii)] at least one of the two outgoing paths starting from $(x,t)$
intersects the region $[-M_1, M_1]^2$.
\end{itemize}
We will show that there exists $M = M(\epsilon, \omega) \geq M_1$ such that 
the region $[-M, M]^2$ contains the set $\Xi^\epsilon$. We will present 
a sketch of the proof here. 

First we Consider outgoing paths starting from $(1,2)$-type points from 
left of the box $[-M_1, M_1]^2$. For $j\geq 1$ we define the rectangle 
$R_j $ as $R_j := [- M_1 - j-1 , - M_1 - j ]\times[-M_1,M_1]$. 
A non-negative integer valued r.v. $\tilde{N}_1$ is defined as 
$$
\tilde{N}_1 = \sup_j \{ j  : R_j \text{ contains at least }j^2 \text{ many }\Xi^\epsilon 
\text{ points}\}.
$$
We need to show that $\tilde{N}_1$ is a.s. finite. Towards that for $j \geq 1$ we define the event:
\begin{align*}
B^1_j := \{\text{The set } R_j
\text{ contains at least }j^2 \text{ many }\Xi^\epsilon \text{ points}\}. 
\end{align*}     
For $j \geq 1$, the total number of $\Xi^\epsilon$ points in $R_j$ is dominated by 
the total number of $(1,2)$ type points in $R_j$. The total number of $(1,2)$ 
type points in $R_j$ is of finite expectation and hence, 
Markov's inequality gives us $\sum_{j \geq 1} \P(B^1_j) < \infty$. By applying Borel Cantelli lemma 
we conclude that $\tilde{N}_1$ is finite a.s.  

Next, we define the r.v. $\tilde{N}_2$ as 
\begin{align*}
\tilde{N}_2 = \sup \{ & j \geq 1  : \text{ the box }R_j \text{ contains at least } 
\text{ one }\Xi^\epsilon \text{ point with at 
least one  } \\ 
& \text{ of the outgoing paths intersects }[-M_1,M_1]^2 \}.
\end{align*}
In order to prove Lemma \ref{lem:CmptBox} it suffices to show that 
the r.v. $\tilde{N}_2$ is also a.s. finite. Towards that, for $j \geq 1$ we define
the event
\begin{align*}
B_j^2 =&\{ j \geq 1 : \text{ at least one of the outgoing paths starting from an element of }\\
 &\Xi^\epsilon \text{ in }R_j \text{ intersects }[-M_1,M_1]^2 \}.
\end{align*}
For $k \in \{0,1, 2, \cdots\}$ let $\Omega_k$ denote the event that 
$ \Omega_k := \{\tilde{N}_1 = k\} $. Note that the collection $\{\Omega_k : k \geq 0\}$
gives a partition of $\Omega$. For any $k \geq 0$ we show that on the event $\Omega_k$ we have
\begin{align}
\label{eq:BC_bound}
\sum_{j=1}^\infty \P(B^2_j \cap \Omega_k)< \infty.
\end{align}
By Borel Cantelli lemma, Equation (\ref{eq:BC_bound}) implies that 
on the event $\Omega_k$ the r.v. $\tilde{N}_2$ is a.s. finite. As $\{\Omega_k : k \geq 0\}$
forms a partition of  the whole space, this completes the proof.
  
The argument for showing (\ref{eq:BC_bound}) is standard and we only
give a sketch here. With a slight abuse of notation let $B(t)_{t \geq 0}$ denote a Brownian 
motion with drift $1$. Application of union bound together with 
translation invariance nature of our model allows us to bound
the probability $\P(B^2_j \cap \Omega_k)$ for any $j > k$ as 
$$
\P(B^2_j \cap \Omega_k) \leq j^2 \P(\sup_{t \in [0,2M_1]} |B(t)| > j).
$$
As $j \to \infty $ the probability $\P(\sup_{t \in [0,2M_1]} |B(t)| > j)$ decays 
exponentially in $j$ and hence, we have $\sum_{j =1}^\infty \P(B^2_j \cap \Omega_k) < \infty $
for all $k \geq 0$.  This completes the proof.
\end{proof}
 
Next, we show that the set $\Xi^\epsilon$ must be finite a.s. 
Let $\epsilon^\prime = \epsilon^\prime(\epsilon, \omega)$ be 
an a.s. strictly positive random variable defined as 
\begin{equation}
\label{eqn:epsilon^prime_choice}
\epsilon^\prime  := \sup\{ 0 <\delta < \tilde{\epsilon} : 
|s_1 - s_2| < \delta \text{ implies that }|\tanh(s_1) - \tanh(s_2)| < \tilde{\epsilon}\},
\end{equation} 
where $\tilde{\epsilon}$ is as in (\ref{eq:epsilon^tilde_choice}).
We are now ready to state the next lemma.

\begin{lemma}
\label{lem:SeparPtContains}
Fix $\epsilon > 0$. Set $\epsilon^\prime = \epsilon^\prime(\epsilon, \omega) > 0$ as in (\ref{eqn:epsilon^prime_choice}) and $M = M(\omega, \epsilon)$ as in Lemma \ref{lem:CmptBox}.
 For all $0 \leq \tau_1 < \tau_2 \leq 1$ a.s. we have 
\begin{align*}
\Xi^\epsilon_{[\tau_1, \tau_2]} \subseteq 
\Bigl ( \bigl ( \cup_{j=0}^{\lfloor 2M / \epsilon^\prime \rfloor} S^{- M + 
(j+1)\epsilon^\prime}_{ - M + j\epsilon^\prime} 
\bigr ) \cap [-M, M]^2 \Bigr ),
\end{align*}
where for $T_1 < T_2$ the set $S^{T_2}_{T_1}$ is defined as in Definition \ref{def:NRS_separation}.
\end{lemma} 
\begin{proof}
As argued earlier, it is enough to prove Lemma \ref{lem:SeparPtContains} 
for $\Xi^\epsilon = \Xi^\epsilon_{[0,1]}$. Because of Lemma \ref{lem:CmptBox}, 
it suffices to show that any $(x,t)$ in $\Xi^\epsilon$ must 
belong to $S^{- M + (j+1)\epsilon}_{ - M + j\epsilon}$ for some $0 \leq j \leq 
\lfloor 2M / \epsilon\rfloor $. By definition, for any $(x,t) \in \Xi^\epsilon$, 
we must have $\text{A}^{(x,t)}_{s_1} > \tilde{\epsilon}$ for some $s_1 \in [0,1]$. 
The choice of $\epsilon^\prime$ ensures that in the corresponding Brownian net $\NN$, 
there must be an incoming path which starts before 
$t - \epsilon^\prime$ and passes through $(x,t)$. 
On the other hand, Lemma \ref{lem:PathAge} ensures that 
we also have $T^{(x,t)}_{[\tau_1, \tau_2]} > \tilde{\epsilon}$.

This ensures that in the corresponding Brownian net $\NN$, 
there are two outgoing paths which start before time
$t - \epsilon^\prime$, get separated at $(x,t)$ and thereafter, they are 
not allowed to meet before time $t + \epsilon^\prime$. 
 Finally,   for any deterministic $s$, the set $\R \times \{s\}$ can not contain any
$(1,2)$ type points. So for any $(x,t) \in \Xi^\epsilon$, there must exist $0 \leq j^\prime \leq 
\lfloor 2M/ \epsilon^\prime \rfloor$ such that $t \in (-M + j^\prime\epsilon^\prime, 
-M + (j^\prime + 1)\epsilon^\prime)$ and the outgoing paths remain separated till time 
$-M + (j^\prime + 1)\epsilon^\prime$. Together with Lemma \ref{lem:CmptBox}, 
this completes the proof.
\end{proof}
For any interval $A \subset [0, 1]$ we consider the set $\Xi^\epsilon_{A}$
and the set of corresponding Poisson clock rings for switching events is denoted as 
 $\Lambda^\epsilon_{A}$. Clearly, $\Lambda^\epsilon_{A}$
 is a random subset of  $ A \subseteq  [0,1]$. 
 For ease of notation the set $\Lambda^\epsilon_{[0, 1]}$ is simply denoted 
 as $\Lambda^\epsilon$. Clearly, for any $A \subseteq [0,1]$ we have 
$$
\Xi^\epsilon_{A} \subseteq \Xi^\epsilon\text{ as well as }
\Lambda^\epsilon_{A} \subseteq \Lambda^\epsilon.
$$ 
Since $S^{- M + (j+1)\epsilon^\prime}_{ - M + j\epsilon^\prime} $ 
are locally finite (see Proposition 7.9 of \cite{NRS10})
Lemma \ref{lem:SeparPtContains} gives  us that the set $\Xi^\epsilon$ is 
{\it finite} almost surely . This implies that 
that the set $ \Lambda^\epsilon $ is a.s. finite as well.
The next lemma shows that for any $0 \leq \tau_1 < \tau_2 \leq 1$
in order to have $\omega_{\WW}([\tau_1, \tau_2)) > \epsilon$
 the set $\Lambda_{[\tau_1, \tau_2)}\cap (\tau_1, \tau_2)$ must be non-empty.
\begin{lemma}
\label{lem:PoissonClockSetIntersection}
For $ 0 \leq \tau_1 < \tau_2 \leq 1$ and for $\epsilon > 0$, on the 
event $\Lambda^\epsilon_{[\tau_1, \tau_2)}\cap (\tau_1, \tau_2) = \emptyset $
we must have $\omega_{\WW}([\tau_1, \tau_2)) \leq \epsilon$.
\end{lemma}
\begin{proof} 
We first observe that in order to have $\omega_{\WW}([\tau_1, \tau_2)) > \epsilon$, 
there must be dynamic times $s_1, s_2 \in [\tau_1, \tau_2)$ with 
$s_1 < s_2$ and a skeletal Brownian path $\pi$ (i.e., starting from a point in $\Q^2$)
such that 
\begin{equation}
\label{eq:SkeletalPath_1}
d_{{\cal H}}(\{\pi_{s_1}\}, \WW(s_2)) > \epsilon,
\end{equation}
where $\pi_{s_1}$ denotes the trajectory of the skeletal path $\pi$ at dynamic time $s_1$.
A skeletal path starts from a point in $\Q^2$ (which is not a $(1,2)$ type point a.s.)
and hence starting time of a skeletal path does not change over dynamic time interval. 
Therefore the starting time of a skeletal path $\pi$ 
does not depend on dynamic time $s_1$ and denoted by $\sigma_{\pi}$. 

Equation (\ref{eq:SkeletalPath_1}) necessarily implies that $d_{\Pi}(\pi_{s_1}, \pi_{s_2}) > \epsilon$
where $\pi_{s_2}$ denotes the trajectory of the same skeletal path $\pi$ at dynamic time $s_2$.
Hence, in order to have (\ref{eq:SkeletalPath_1}), there must be a separating 
point $(x,t)$ on the trajectory of $\pi_{s_1}$, as a point of separation between $\pi_{s_1}$ and $\pi_{s_2}$, such that 
$$
D^{(x,t)}_{[\tau_1, \tau_2)} \geq D^{(x,t)}_{[s_1, s_2]} > \epsilon.
$$
In other words, in order to have (\ref{eq:SkeletalPath_1}), the set 
$D^\epsilon_{[\tau_1, \tau_2)}$ must be non-empty. We observe 
that for any $(1,2)$ type point $(y,s)$ to be in
the set $D^\epsilon_{[\tau_1, \tau_2)}$, the associated Poisson clock 
must ring during the interval $(\tau_1, \tau_2)$. Recall that the set 
$\Lambda^\epsilon_{[\tau_1, \tau_2)}$ 
represents Poisson clock rings associated to  points in the set 
$\Xi^\epsilon_{[\tau_1, \tau_2)} = D^\epsilon_{[\tau_1, \tau_2)}\cap 
A^{\tilde{\epsilon}}_{[\tau_1, \tau_2)}$,  
where $\tilde{\epsilon}$ is as in (\ref{eq:epsilon^tilde_choice}). 

Let us assume that $ \Lambda^\epsilon_{[\tau_1, \tau_2)}\cap (\tau_1, \tau_2) = \emptyset $.
By Remark \ref{rem:EtaSet} this event equivalently implies that the set 
$ \Xi^\epsilon_{[\tau_1, \tau_2)}$ is empty as well. 
Further, the condition  $ \Xi^\epsilon_{[\tau_1, \tau_2)} = \emptyset $
 implies that for any $(y,s) \in D^\epsilon_{[\tau_1, \tau_2)}$, we must 
have $A^{(y,s)}_{[\tau_1, \tau_2)} \leq \tilde{\epsilon}$.
We will show that under this assumption, we can't have (\ref{eq:SkeletalPath_1})
and henceforth we obtain a contradiction.

Choose $(y,s)$ to be a (random) point on the trajectory of $\pi_{s_1}$ 
such that $\tanh(s) - \tanh(\sigma_\pi) \in (\tilde{\epsilon}, 2\tilde{\epsilon})$
and $(y,s)$ is {\it not} a $(1,2)$ type point. Under our 
assumption  $ \Xi^\epsilon_{[\tau_1, \tau_2]} = \emptyset $, the trajectory of
$\pi_{s_1}$ restricted over time domain $[s, \infty)$ can not have a point 
in the set $D^\epsilon_{[\tau_1, \tau_2)}$. As $(y,s)$ is not a $(1,2)$ type point,
there exists a unique path  $\pi^\prime$ in $\WW(s_2)$ starting from the point $(y,s)$. 
Our assumption also ensures that there is no point in $D^\epsilon_{[\tau_1, \tau_2)}$
 which belongs to the trajectory of $\pi^\prime$. Finally, as we have $\tanh(\sigma_{\pi^\prime}) - \tanh(\sigma_{\pi}) < 2 \tilde{\epsilon}$, we obtain
$$
d_\Pi(\pi_{s_1}, \pi^\prime) \leq \epsilon.
$$
This completes the proof. 
\end{proof}

Now we are ready to complete the proof of Proposition \ref{prop:DyBWModCont}
and thereby proving that the DyBW is RCLL a.s. 

{\bf Proof of Proposition \ref{prop:DyBWModCont}:} Fix $\epsilon > 0$. 
We choose $\delta = \delta(\epsilon, \omega) > 0$ such that for any $s_1, s_2 \in 
\Lambda^\epsilon(\omega)$ we have 
$$
\min\{ s_1, (1-s_2), |s_1 - s_2| \} > 2\delta.
$$
With the above choice of $\delta$,  there exists a partition of the interval 
$[0,1]$ as $ 0 = s^\prime_0 < s^\prime_1 < \cdots < s^\prime_m = 1$ such that 
$s^\prime_{i+1} - s^\prime_{i} > \delta$ and 
$ \Lambda^\epsilon \cap (s^\prime_i, s^\prime_{i+1}) = \emptyset$ for all 
$ 0 \leq i \leq m - 1 $.
This readily implies that $ \Lambda^\epsilon_{[s^\prime_i, s^\prime_{i+1})} \cap 
(s^\prime_i, s^\prime_{i+1}) = \emptyset$ for all $\ 0 \leq i \leq m - 1 $. Hence,   
Lemma \ref{lem:PoissonClockSetIntersection} ensures that we have
\begin{align*}
\max\{\omega_{\WW}([s^\prime_i, s^\prime_{i+1})): 0 \leq i \leq m - 1\} 
\leq \epsilon.
\end{align*} 
This completes the proof.
\qed

\section{Finite dimensional distribution convergence of the DyDW to the DyBW}
\label{sec:FDD_DyDW}

In this section we describe finite dimensional convergence of the DyDW 
to the DyBW. As commented earlier, the main argument for the same was already done 
in \cite{NRS09A} and we have followed the same strategy.
We first start with describing the dynamic discrete web model. 
After that we describe finite dimensional distributions of the DyBW and 
in Section \ref{subsec:FDD} we present the required finite dimensional
convergence.   

\subsection{Dynamic discrete web (DyDW) model}

The discrete web (DW) is a system of coalescing simple symmetric one-dimensional random walks 
starting from everywhere on the space time even lattice 
$\Z^2_{\text{even}} := \{(x,t) \in \Z^2 : x + t \text{ even}\}$. 
We described in the beginning that we have an i.i.d. collection of 
Rademacher random variables $\{I_{(x,t)} : (x,t) \in \Z^2_{\text{even}}\}$ where
 $I_{(x,t)}$ gives increment of the walker at location $x$ at time $t$.
Following the walk, the walker at $(x,t)$ reaches $(x + I_{(x,t)}, t+1)\in \Z^2_{\text{even}}$
and this next step is denoted as $h(x,t)$. Set $h^0(x,t) = (x,t)$
and for $k \geq 1$, define $h^k(x,t) = h(h^{k-1}(x,t))$. Joining the successive steps 
$\langle h^{k-1}(x,t), h^{k}(x,t) \rangle$ we get the path $\pi^{(x,t)} \in \Pi$ starting from 
$(x,t)$. The collection of all paths obtained from the DW is denoted as 
${\cal X} := \{\pi^{(x,t)} : (x,t) \in \Z^2_{\text{even}}\}$. 
For $n \in \N$ and for any $(x,t) \in \R^2$, let $S_n$ denote the 
 $n$-th order diffusively scaled map given by 
$$
S_n(x,t) := (x/\sqrt{n}, t/n).
$$
With a slight abuse of notation for $A \subset \R^2$ we denote $S_n(A)$
as $S_n(A) := \{(x/\sqrt{n}, t/n) : (x,t) \in A\}$. 
For $\pi \in \Pi$ and for ${\cal K} \subset \Pi$, we define $S_n(\pi)$ and $S_n({\cal K})$
by identifying each path with it's graph as a subset of $\R^2$.
For any $n \geq 1$ closure of the set $S_n({\cal X})$ taken in $(\Pi, d_\Pi)$ metric is denoted as 
$\overline{S_n({\cal X})}$. This gives a ${\cal H}$ valued random variable.
Fontes et. al. \cite{FINR04} proved that as a sequence of ${\cal H}$ valued 
random variables, $\{\overline{S_n({\cal X})} : n \in \N\}$ converges in distribution to the Brownian web. 
\begin{theorem}[Theorem 5.1 of \cite{FINR04}]
\label{thm:DW_BW}
As $n \to \infty$, $\overline{S_n({\cal X})}$ converges in distribution to the Brownian web $\WW$.
\end{theorem}

The DyDW is a dynamic evolution of the discrete web where 
 the arrow configuration evolves continuously over   dynamic time
 such that at each space-time point $(x,t) \in \Z^2_{\text{even}}$, 
the outgoing arrow $\langle (x,t), h(x,t) \rangle$ switches at unit rate independent 
of everything else. In order to define the DyDW we consider independent sequences of  
collections of i.i.d. random variables $\{I^m_{(x,t)} : (x,t) \in \Z^2\}_{m \geq 0}$.  
For each $(x,t) \in \Z^2$ we consider a Poisson point process $\{ N_{(x,t)}(\tau) : 
\tau \in [0,1]\}$ which is a Markov process with state space 
$\{0\} \cup \N$ evolving over dynamic time
$[0,1]$ with unit intensity.  
In other words, for any $\tau \in [0,1]$ we have 
$$
\P( N_{(x,t)}(\tau)  = \ell) = \frac{e^{-\tau} \tau^\ell}{\ell \times \cdots \times 1} \text{ for }\ell \geq 0.
$$
We further assume that these Markov processes,
 as $(x,t)$ varies over $\Z^2_{\text{even}}$, 
are mutually independent of each other. At  dynamic time
$\tau \in [0,1]$, for each $(x,t)$ we use the increment $I^{N_{(x,t)}(\tau)}_{(x,t)}$
and this gives us a distributionally equivalent 
copy of the DW. Altogether, this represents the DyDW as 
an ${\cal H}$-valued stationary Markov process evolving 
over dynamic time domain $[0,1]$.

Mathematically, for $(x,t) \in \Z^2_{\text{even}}$ and dynamic time $\tau$, 
next step at that instant is denoted as $h_\tau(x,t)$ and given as 
$$
h_\tau(x,t) := (x + I^{N_{(x,t)}(\tau)}_{(x,t)}, t+1) \in \Z^2_{\text{even}}. 
$$
$h^k_\tau(x,t)$, i.e., the $k$-th step at dynamic time instant $\tau$ is similarly defined. 
The path $\pi^{(x,t)}_\tau \in \Pi$ starting from $(x,t)$ at dynamic time $\tau$
is obtained by linearly joining successive steps 
$\langle h^{k-1}_\tau (x,t), h^{k}_\tau (x,t) \rangle$. The collection of all paths 
obtained from DyDW at dynamic time $\tau$ is denoted as 
${\cal X}(\tau) := \{\pi_\tau^{(x,t)} : (x,t) \in \Z^2_{\text{even}}\}$
and $\{ \overline{{\cal X}}(\tau) : \tau \in [0,1]\}$ represents the DyDW process defined over the
dynamic time domain $[0,1]$.

Fix  $n \in \N$ and now we are going to define the $n$-th diffusively 
scaled DyDW. Let $\{N^n_{(x,t)}(\tau) : \tau \in [0,1]\}_{(x,t) \in \Z^2}$ 
denote a collection of independent Poisson processes with intensity $1/\sqrt{n}$, i.e., 
$$
\P( N^n_{(x,t)}(\tau)  = \ell) = 
\frac{e^{-\tau/\sqrt{n}}(\tau/\sqrt{n})^{\ell}}{\ell \times \cdots \times 1} \text{ for }\ell \geq 0.
$$
Corresponding to these Markov processes, we define
$$
h_{\tau, n}(x,t) := (x + I^{N^n_{(x,t)}(\tau)}_{(x,t)}, t+1)
$$
 as the next step at dynamic time $\tau$. The path $\pi^{(x,t),n}_{\tau}$ at dynamic time $\tau$
is obtained by linearly joining successive steps 
$\langle h^{k-1}_{\tau,n} (x,t), h^{k}_{\tau,n} (x,t) \rangle$.
Let ${\cal X}^n(\tau) := \{\pi^{(x,t),n}_{\tau} : (x,t) \in \Z^2_{\text{even}}\}$.
Let ${\cal X}_n(\tau)$ denote the $n$-th scaled collection of paths $S_n({\cal X}^n(\tau))$
and the $n$-th scaled DyDW process is denoted by $\{\overline{{\cal X}_n}(\tau) : \tau \in [0,1]\}$.

In this paper we show that as the scaled DyDW as a process with RCLL paths converges to 
the DyBW. 
\begin{theorem}
\label{thm:DYDW_Conv}
As $n \to \infty$, the DyDW process 
$\{\overline{{\cal X}_n}(s) : s \in [0,1]\}$ converges to the DyBW
$\{ \WW(s) :s \in [0,1]\}$ where the convergence happens as 
${\cal D}^{{\cal H}}([0,1])$ valued random variables. 
\end{theorem}

In the next section we show that 
the finite dimensional distributions of the DyDW process converge to that of the DyBW.
A detailed sketch of the argument for the same was earlier presented in \cite{NRS09A}.
But to the best of our knowledge, weak convergence of the DyDW process 
$\{\overline{{\cal X}_n}(s) : s \in [0,1]\}$  to the DyBW process has not been shown so far.
In Section \ref{sec:Tight} we prove tightness of the scaled DyDW
and thereby complete the proof of Theorem \ref{thm:DYDW_Conv}.  


\subsection{Finite dimensional distribution convergence}
\label{subsec:FDD}
In this section we prove the following proposition 
which proves finite dimensional distributional convergence for the scaled DyDW to the DyBW. 
\begin{proposition}
\label{prop:FDD_Conv}
Fix deterministic $0 \leq \tau_1 < \tau_2 < \cdots < \tau_k  \leq 1$. 
Then as $n \to \infty$, we have 
$$
(\overline{{\cal X}}_n(\tau_1), \cdots, \overline{{\cal X}}_n(\tau_k)) \Rightarrow
(\WW(\tau_1), \cdots, \WW(\tau_k)),
$$ 
as ${\cal H}^k$ valued random variables.
\end{proposition}
In the next section, we describe finite dimensional distributions of the DyBW process. 
We prove Proposition \ref{prop:FDD_Conv} after that.

\subsubsection{Finite dimensional distributions of the DyBW}
\label{subsec:Sticky_BW}

We first recall the definition of a one-dimensional sticky (at the origin) Brownian motion.
\begin{definition} 
$B_{\text{stick},x}$ is a $(1/\tau)$-sticky Brownian motion starting at $x$ iff there exists a one-dimensional standard Brownian motion $B$ s.t.
\begin{equation}
\label{eq:Sticky_BM}
dB_{\text{stick},x}(t) = \mathbf{1}_{B_{\text{stick},x}(t) \neq 0}dB(t) +
\tau \mathbf{1}_{B_{\text{stick},x}(t) = 0}dt \text{ for all }t \geq 0,
\end{equation}
and $B$ is constrained to stay positive as soon as it hits zero.
\end{definition}
It is known that (\ref{eq:Sticky_BM}) has a unique (weak) solution. For $x = 0$
 this solution can be constructed from a
time-changed reflected Brownian motion as follows. Consider
\begin{equation}
\label{eq:Sticky_BM_1}
t \mapsto |\overline{B}|(C(t)) \text{ with } C^{-1}(t) = t + \frac{1}{\tau}L_0(t),
\end{equation}
where $|\overline{B}|$ is the reflected Brownian motion and $L_0(t)$ is its local time 
at the origin. Then there exists a Brownian motion $\overline{B}$ such that 
$(|\overline{B}|(C(\cdot)),B)$ is a solution of (\ref{eq:Sticky_BM}). 
The sticky Brownian motion is obtained from the reflected Brownian motion 
by ``transforming" the local time into real time and as a result, 
 it spends positive Lebesgue measure time at the origin. The larger the 
 ``degree of stickiness"  $1/\tau$ is, the more the path sticks to the origin. 

We now describe finite dimensional distributions of the DyBW. 
We first present the definition of sticky pair of Brownian motions 
starting from any two points in $\R^2$ taken from \cite{NRS10}.
\begin{definition} 
$(B,B^\prime)$ is a $(1/\tau)$-sticky pair of Brownian motions iff:
\begin{itemize}
\item[(i)] $B$ and $B^\prime$ are both Brownian motions starting at $(x_B, t_B)$
 and $(x_{B^\prime}, t_{B^\prime} )$ that move independently when they do not
coincide.
\item[(ii)] For $t \geq 0$, define $B_{\text{stick}}(t) := |B - B^\prime|(t + t_B \vee t_B^\prime )/\sqrt{2}$. Conditioned on $x = B_{\text{stick}}(0)$, 
the process $\{ B_{\text{stick}}(t) : t\geq 0 \} $ is a $(\sqrt{2}/\tau )$- sticky Brownian motion 
starting at $x$ (see Definition \ref{eq:Sticky_BM}).
\end{itemize}
\end{definition}
We call $(B_1, \cdots , B_m ; B^\prime_1 , \cdots , B^\prime_n)$ a collection of $(1/\tau )$-sticking-coalescing Brownian motions, if $(B_1, \cdots , B_m)$ and $(B^\prime_1 , \cdots , B^\prime_n)$ are each distributed as a set of coalescing Brownian motions and for any $B \in \{ B_1, \cdots , B_m \}$ and
$B^\prime \in \{ B^\prime_1 , \cdots , B^\prime_n \}$, the pair $(B,B^\prime )$ is a $(1/\tau )$-sticky pair of Brownian motions.

We will say that $(\WW,\WW^\prime )$ is a $1/\tau$-sticky pair of Brownian 
webs if $(\WW,\WW^\prime)$ satisfies the following properties:
\begin{itemize}
\item[(a)] $\WW$, resp. $\WW^\prime$, is distributed as the standard Brownian web.
\item[(b)] For any finite deterministic set $\x_1, \cdots , \x_m, \x^\prime_1, \cdots , 
\x^\prime_n  \in \R^2$, the subset of paths in $\WW$ starting from these points
 are jointly distributed as a collection of $(1/\tau )$-sticking-
coalescing Brownian motions starting from the given sets of points.
\end{itemize}
Newman et. al. \cite{NRS10} presented a rigorous construction of 
the process $\{\WW(\tau) : \tau \in [0,1]\}$ such that for any $0 \leq \tau_1 < \tau_2 \leq 1$ 
the pair $(\WW(\tau_1), \WW(\tau_2))$ equidistributed as $(\WW(0), \WW(\tau_2 - \tau_1))$
which has the same distribution as $1/(2(\tau_2 - \tau_1))$-sticky pair of Brownian webs.
Our next remark explains that the distribution of $(\WW(\tau_1), \WW(\tau_2))$
uniquely specifies finite dimensional distributions of the DyBW process. 
Remark \ref{rem:FDD_k_DyBW} essentially follows from stationarity and 
Markov property of theorem 6.2 of \cite{NRS10}.
\begin{remark}
\label{rem:FDD_k_DyBW}
Set $k \geq 1$ and fix $0\leq \tau_1 < \cdots < \tau_k \leq 1$. Let $({\cal Z}_1, \cdots , 
{\cal Z}_k)$ be such that  for all $ 1 \leq i \leq k $, we have
\begin{equation}
\label{eq:Dist_W}
{\cal Z}_i \stackrel{d}{=} \WW,
\end{equation}
where the notation $\stackrel{d}{=}$ stands for same distribution.
 Further, for any $k$  many deterministic points $(x_1, t_1), 
\cdots , (x_k, t_k) \in \R^2$, let $\pi^i$ denote the a.s. unique path in ${\cal Z}_i$
starting from $(x_i, t_i)$. Equation (\ref{eq:Dist_W}) guarantees existence 
and uniqueness of such $\pi^i$. Suppose distribution  
$(\pi^1, \cdots , \pi^k) \in \Pi^k$ satisfies the following:
\begin{itemize}
\item[(a)] For each $1 \leq i \leq k$, marginally  the random path $\pi_i$
is distributed as the standard Brownian motion starting from $(x_i, t_i)$.  
\item[(b)] As long as the paths are disjoint, they evolve like independent Brownian motions.  
\item[(c)] As soon as any two paths intersect, due to sticky interaction they 
spend non-trivial time together. Precisely, if paths $\pi^i$ and $\pi^j$ intersect, 
 together they evolve like a $1/(2|\tau_i - \tau_j|)$-sticky Brownian motion.    
\end{itemize}
Then we must have 
$$
({\cal Z}_1, \cdots , {\cal Z}_k) \stackrel{d}{=}
(\WW(\tau_1), \cdots , \WW(\tau_k)),  
$$
where $\WW(\tau)$ denotes the DyBW at dynamic time $\tau$.
\end{remark}  

In the next section we use Remark \ref{rem:FDD_k_DyBW} to 
prove Proposition \ref{prop:FDD_Conv}.

\subsubsection{Proof or Proposition \ref{prop:FDD_Conv}}

\noindent {\bf Proof of Proposition \ref{prop:FDD_Conv}:}
Note that the sequence $\bigl\{ \bigl ( \overline{{\cal X}}_n(\tau_1), 
\cdots , \overline{{\cal X}}_n(\tau_k) \bigr ) : n \in \N\bigr \}$
is a tight sequence of ${\cal H}^k$-valued random variables. 
Consider any sub-sequential limit $({\cal Z}_1,\cdots, {\cal Z}_k )$ of the above sequence. 
The work of Fontes et al. \cite{FINR04} ensures that
$$
{\cal Z}_i \stackrel{d}{=} \WW \text{ for all }1 \leq i \leq k.
$$  
Because of Remark \ref{rem:FDD_k_DyBW}, in order to prove Proposition \ref{prop:FDD_Conv}
it suffices to show that $({\cal Z}_1,\cdots, {\cal Z}_k )$ satisfies conditions (b) and (c) as well. 
Condition (a) follows from the fact that, as long as the scaled discrete web paths are disjoint, 
they evolve independently. 

To show condition (a) we fix $1 \leq i, j \leq k$. 
Set $(x_i,t_i) = (x_j,t_j) = (0,0)$, i.e., $i$-th and $j$-th path both start 
from origin. The DW paths evolve independently as long as they are supported on disjoint sets of lattice 
points and hence, taking two different starting points does not pose any additional challenge.   
For each $n \geq 1$, the collection $\{ N^n_{(x,t)}(\tau) : \tau \in [0,1]\}_{(x,t) 
\in \Z^2_{\text{even}}}$ denotes a collection of i.i.d. Poisson processes with 
intensity $1/\sqrt{n}$ and these processes are taken to be independent. 
Corresponding to this, $\pi^{\mathbf{0},n}_{\tau_i}$ and $\pi^{\mathbf{0},n}_{\tau_j}$
denote the unscaled DW  paths starting from the origin at dynamic times $\tau_i$ and $\tau_j$ 
respectively. We comment here that the marginal distribution of $\pi^{\mathbf{0},n}_{\tau_i}$ does 
not depend on $n$ but the joint distribution of $(\pi^{\mathbf{0},n}_{\tau_i}, \pi^{\mathbf{0},n}_{\tau_j})$
depends on $n$. To simplify our notation, we denote these two paths simply as 
$S^{i,n}$ and $S^{j,n}$ respectively. W.l.o.g. we assume that $\tau_i < \tau_j$.
Corresponding pair of $n$-th order diffusivelly scaled paths are given by 
$$
(S^i_n(t), S^j_n(t))_{t \geq 0}
:= (S^{i,n}(nt)/\sqrt{n}, S^{j,n}(nt)/\sqrt{n})_{t \geq 0}.
$$

For this part of the proof we follow the ideas in \cite{NRS09} and in \cite{NRS09A}.
We observe that the pair of `scaled' DW paths $S^i_n$ and $S^j_n$ alternates between 
times at which these two paths are equal (i.e., they “stick together”) and times at 
which they move independently. As soon as the paths $S^i_n$ and $S^j_n$ meet at 
time $t = k/n$ for some $k \in \N$, they continue to coincide and move together as long as the clock at 
$(\sqrt{n}S^{i}_{n}(k^\prime / n),k^\prime) = (\sqrt{n}S^{j}_{n}(k^\prime/n),k^\prime)$ 
for some $k^\prime \in \{ k+1, k+2, \cdots \}$ 
does not ring during the interval $(\tau_i, \tau_j]$. After such a clock ticks, 
the two random walk paths are independent till the time they meet again.
This suggests the following time decomposition.  Set $T^n_0=0$ and for $k \geq 0$ we define
\begin{align*}
T^n_{2k+ 1} & := \inf\{ m \geq T^n_{2k}, m \in \N : \text{ the Poisson clock } 
N^n_{(\sqrt{n}S^{i}_{n}(m/n), m)} \text{ ticks in }(\tau_i, \tau_j] \}\text{ and }\\
T^n_{2k+ 2} & := \inf\{ m > T^n_{2k + 1}, m \in \N : S^{i}_{n}(m/n) = S^{j}_{n}(m/n) \}.
\end{align*}
As discussed before, on the interval $[T^n_{2k}/n, T^n_{2k+1}/n]$, both the scaled paths 
$S^{i}_{n}$ and $S^{j}_{n}$ coincide and at time $T^n_{2k+1}/n$ they start moving 
independently until meeting at $T^n_{2k+2}/n$.
In other words, if we skip the intervals
$[T^n_{2k}/n, T^n_{2k+1}/n)_{k\geq 0}$, the pair $(S^{i}_{n}, S^{j}_{n})$ behave 
as two independent diffusively scaled random walks $(\tilde{S}^i_{n},  \tilde{S}^j_{n})$, while if we skip intervals
$[T^n_{2k+1}/n, T^n_{2k+2}/n]_{k\geq 0}$, the two walks coincide to form a single diffusively scaled 
random walk $\tilde{S}^s_n$ which  evolves independent of $(\tilde{S}^i_{n}, \tilde{S}^j_{n})$.
We define $\Delta T^n_k := T^n_{2k + 1} - T^n_{2k}$ and we have 
$$
\P(\Delta T^n_k \geq \ell )\leq \exp{(-\ell\frac{|\tau_j - \tau_i|}{\sqrt{n}})}.
$$ 
Further, the collection $\{ \Delta T^n_k : k \geq 0\}$ is independent 
of the rescaled random walk paths.
This analysis, which  is a scaled version of Lemma 3.2 of \cite{NRS09}, is summarised as follows:

Distributionally this scaled pair is equivalent to 
\begin{align*}
S^i_n(t) & = \tilde{S}^i_n(C_n(t)) + \tilde{S}^s_n(t - C_n(t))\text{ and }\\ 
S^j_n(t) & = \tilde{S}^j_n(C_n(t)) + \tilde{S}^s_n(t - C_n(t)),
\end{align*} 
where $\tilde{S}^i_n, \tilde{S}^j_n, \tilde{S}^s_n$ are three independent 
rescaled random walks, and $C_n$ is the right continuous inverse of 
$\hat{L}_n(t) + t$ with,
\begin{itemize}
\item[(i)] $\hat{L}^n(t) := \frac{1}{n}\sum_{k=1}^{\lfloor \hat{l}^n(t) \sqrt{n} \rfloor} T^n_k$;
\item[(ii)] $\hat{l}^n(t) := \frac{1}{\sqrt{n}}\#\{ k \leq nt : \tilde{S}^i_n(k/n) = 
\tilde{S}^j_n(k/n)\}$,   
\end{itemize}
where  $T^n_k$'s are i.i.d. non-negative random variables with 
$\P (T^n_k \geq \ell) = \exp{(-\ell \frac{|\tau_i - \tau_j|}{\sqrt{n}})}$.
 
Then we have 
$$
\E(T^n_k) = \frac{\exp{(-\frac{|\tau_i - \tau_j|}{\sqrt{n}})}}{1 - 
\exp{(-\frac{|\tau_i - \tau_j|}{\sqrt{n}})}}\text{ and }Var(T^n_k) = 
\frac{\exp{(-\frac{|\tau_i - \tau_j|}{\sqrt{n}})}}{\bigl (1 - 
\exp{(-\frac{|\tau_i - \tau_j|}{\sqrt{n}})} \bigr )^2}.
$$
We know that $\hat{l}^n(t)$ converges in distribution to $L(t)$ where $L(t)$ is 
the local time at zero of $(B_1 - B_2)$ where $B_1$ and $B_2$ are
two independent Brownian motions (see Theorem 1.1 of \cite{B82} and Theorem 4.1 of \cite{K63}).

By Skorohod's representation theorem 
we assume that we are working on a probability space 
$(\Omega_l, {\cal B}_l, \P_l)$ such that $\hat{l}^n(t)$
converges almost surely to $L(t)$. 
Note that the scaled independent random walks $\{ (\bar{S}^i_n(s), \bar{S}^j_n(s), \bar{S}^s_n(s) )
: s \geq 0\}$ ${\cal B}_l$ are measurable. Let us denote the probability space 
for the i.i.d. collections $\{ T^n_k : k \geq 1\}_{n \geq 1}$ as 
$(\Omega_d, {\cal B}_d, \P_d)$ and let $\P = \P_l \times \P_d$. 
Then we have 
\begin{align*}
Var(\hat{L}^n(t) \mid {\cal B}_l) & = Var \bigl (  \frac{1}{n}\sum_{k=1}^{\lfloor \hat{l}^n_t \sqrt{n}\rfloor}T^n_k \bigr )\\
& = \frac{1}{n^2}\sum_{k=1}^{\lfloor \hat{l}^n_t\sqrt{n}\rfloor}Var(T^n_k)\\
& = \frac{1}{n^2} \lfloor \hat{l}^n_t \sqrt{n}\rfloor 
\frac{\exp{(-\frac{|\tau_i - \tau_j|}{\sqrt{n}})}}{\bigl (1 - 
\exp{(-\frac{|\tau_i - \tau_j|}{\sqrt{n}})} \bigr )^2} \\
& \equiv O(1/\sqrt{n}).
\end{align*}
Therefore we have $\P_l$ a.s. $Var(\hat{L}^n(t) \mid {\cal B}_l) \to 0$ as $n\to \infty$.
Thus we have $ |\hat{L}^n(t) - \E(\hat{L}^n(t)\mid {\cal B}_l)|$ 
converges to $0$  in Probability.
 Since,
$$
\lim_{n \to \infty} \E(\hat{L}^n(t) \mid {\cal B}_l) = 
\lim_{n \to \infty} \frac{\lfloor \hat{l}^n_t
\sqrt{n}\rfloor}{n(|\tau_i - \tau_j|/\sqrt{n})} = 
 \frac{ \lim_{n \to \infty}\hat{l}^n_t}{|\tau_i - \tau_j|} =
 \frac{ {L}_t }{|\tau_i - \tau_j|}, \P_l \text{ a.s.} 
$$
Therefore, we have $\hat{L}^n(t)$ converges to $L(t)/|\tau_i - \tau_j|$ in probability
which implies convergence in distribution for local times. 
This completes the proof for Proposition \ref{prop:FDD_Conv}.

\section{Tightness part}
\label{sec:Tight}

In this section we finally prove that the scaled DyDW processes 
$\{\overline{{\cal X}}_n(\tau) : \tau \in [0,1]\}$ form a tight sequence 
as ${\cal D}^{{\cal H}}([0,1])$ valued random variables and thereby completes 
the proof of Theorem \ref{thm:DYDW_Conv}.

\subsubsection{Description of notations}

Recall that the $n$-th scaled discrete DyDW process is denoted  as 
$\{\overline{{\cal X}}_n(\tau) : \tau \in [0,1]\}$, a ${\cal H}$-valued stochastic process
evolving over the dynamic time interval $[0,1]$. 
 
 Fix $n \in \N$. We consider the scaled DyDW process $\{\overline{{\cal X}}_n(\tau) : 
\tau \in [0,1]\}$. We denote the random configuration of {\it all} outgoing edges
observed over the dynamic time interval $[0,1]$ as
\begin{align*}
{\cal A}_n & := \{ \langle (x,t), h_{\tau, n}(x,t)\rangle : (x,t) \in \Z^2_{\text{even}}, 
\tau \in [0,1]\}\\
& = \{ \langle (x,t), (x + 1,t+1) \rangle : (x,t) \in \Z^2_{\text{even}}, 
h_{\tau, n}(x,t) = (x+1, t+1) \text{ for some }\tau \in [0,1]\} \\
& \bigcup \{ \langle (x,t), (x - 1,t+1) \rangle : (x,t) \in \Z^2_{\text{even}}, 
h_{\tau, n}(x,t) = (x - 1, t+1) \text{ for some }\tau \in [0,1]\}.
\end{align*}
The random quantity ${\cal A}_n$ basically denotes the random 
configuration of all `arrow's observed over dynamic time domain $[0,1]$
as described in \cite{SS08}. An ${\cal A}_n$-path, is the graph of a
function $\pi : [\sigma_\pi ,\infty] \mapsto \R \cup \{ \star \}$, 
with $\sigma_\pi \in \Z \cup \{ \pm 
\infty \}$, such that $\langle (\pi(t ), t ), (\pi(t +1), t + 1)\rangle \in {\cal A}_n$ and $\pi$ is linear on the interval $[t, t + 1]$ for all $t \in [\sigma_\pi ,\infty] \cap \Z$, while $\pi(\pm \infty) = \times$ whenever 
$\pm \infty \in [\sigma_\pi, \infty]$. 

The closure of collection of all ${\cal A}_n$-paths, i.e., paths along the arrows in ${\cal A}_n$,
is denoted as $\NN_n$. We observe that $\NN_n$ denotes the $n$-th scaled discrete net 
considered in \cite{SS08} and is a ${\cal H}$-valued random variable.
   
More generally, for any  dynamic time interval
$I \subseteq [0,1]$ collection of all outgoing edges
observed over $I$ is denoted as 
$$
{\cal A}_{I,n} :=  \{ \langle (x,t), h_{\tau, n}(x,t)\rangle : 
(x,t) \in \Z^2_{\text{even}}, \tau \in I\}
$$
 and the closure of collection of all 
 ${\cal A}_{I,n}$-paths is denoted as $\NN_{I,n}$.
 

Sun et. al showed that the discrete net $\NN_n$ converges 
in distribution to the standard Brownian net (Theorem 1.1 of \cite{SS08}). 
Their argument also proves that for any deterministic $0 \leq \tau_1 < \tau_2 \leq 1$ we have     
$$
\NN_{[\tau_1, \tau_2],n} \Rightarrow \NN_{[\tau_1, \tau_2]} \text{ as }n \to \infty.  
$$
Fix $\epsilon > 0$. For $n \in \N$ and for any interval $I 
\subseteq [0, 1]$ considering ${\cal H}$-valued random variable 
$\NN_{I, n}$ we define the 
following quantities:
\begin{align*}
{\cal S}_{I, n} & := {\cal S}(\NN_{I, n}),\\
D^\epsilon_{I, n} & := D^\epsilon(\NN_{I, n}),\\
T^\epsilon_{I, n} & := T^\epsilon(\NN_{I, n}) \text{ and }\\
A^\epsilon_{I, n} & := A^\epsilon(\NN_{I, n}).
\end{align*}
For $(x,t) \in \Z^2_{\text{even}}$ a scaled lattice point $(x,t)_n := S_n(x,t)$
may belong to the set ${\cal S}_{I, n}$ only if 
the associated Markov process  $\{ N^n_{(x,t)}(\tau) : \tau \in [0,1]\}$ ticks
at least once during the dynamic time interval $I$.
In fact, for $I = [\tau_1, \tau_2)$, to be in the set ${\cal S}_{[\tau_1, 
\tau_2), n}$, the associated Markov process must tick
at least once during the interval $(\tau_1, \tau_2)$.  
We define the (random) set $\Xi^\epsilon_{I, n}$ as 
$$
\Xi^\epsilon_{I, n} := D^\epsilon_{I, n}
\cap A^{\tilde{\epsilon}}_{I, n}, 
$$
where $\tilde{\epsilon}$ is defined as in (\ref{eq:epsilon^tilde_choice}).
The notation $\Lambda^\epsilon_{I, n}$ denotes the random set of 
Poisson clock rings attached to points in the set $\Xi^\epsilon_{I, n}$
 over dynamic time  interval $I$. For simplicity of notation we set 
\begin{align*}
\Xi^\epsilon_{[0,1],n} & := \Xi^\epsilon_{n} \text{ and } \\
\Lambda^\epsilon_{[0,1],n} & := \Lambda^\epsilon_{n}.
\end{align*}
As observed earlier, for all $I \subseteq [0, 1]$ and any $n \in \N $ we have 
$\Xi^\epsilon_{I, n} \subseteq \Xi^\epsilon_{ n}$
and $\Lambda^\epsilon_{I, n} \subseteq \Lambda^\epsilon_{n}$ a.s.  

\subsubsection{Proof of tightness }

First we need to show that the scaled DyDW process 
$\{\overline{{\cal X}}_n(\tau) : \tau \in [0,1]\}$ has RCLL ${\cal H}$-valued paths a.s.
We will prove a weaker version that for all large $n$, the process
$\{\overline{{\cal X}}_n(\tau) : \tau \in [0,1]\}$ has RCLL paths a.s. 
For $\delta > 0$ with a slight abuse of notation let $\omega_{\overline{{\cal X}}_n}(\delta)$
denote  the (generalised) modulus  of continuity 
for the DyDW process  $\{\overline{{\cal X}}_n(\tau) : \tau \in [0,1]\}$. 
We first show that the same argument of Lemma \ref{lem:PoissonClockSetIntersection}
holds for the scaled DyDW process as well and gives the following lemma.

\begin{lemma}
\label{lem:PoissonClockSetIntersection_1}
Fix $\epsilon > 0$ and $0 \leq \tau_1 < \tau_2 \leq 1$. For all large $n$
on the event  
$\Lambda^\epsilon_{[\tau_1, \tau_2), n} \cap (\tau_1, \tau_2) = \emptyset $ we must 
have $\omega_{\overline{{\cal X}}_n}([\tau_1, \tau_2)) \leq \epsilon$.
\end{lemma}
\begin{proof}
The argument is essentially same as that of  Lemma \ref{lem:PoissonClockSetIntersection} and 
we only give a sketch of the proof here. 
As argued earlier, in order to have $\omega_{\overline{{\cal X}}_n}([\tau_1 , \tau_2)) > \epsilon$ there must be a scaled path $\pi^n_{s_1} \in {\cal X}_n(s_1)$ for some 
$s_1 \in [\tau_1 , \tau_2)$ with a scaled separating (branching) point $(x,t)$ 
on it's trajectory such that $(x,t) \in D^\epsilon_{[\tau_1, \tau_2),n}$. 

We choose a scaled lattice point $(y^n,s^n) \in S_n (\Z^2_{\text{even}})$
on the trajectory of the path $\pi^n_{s_1}$ and above $t$ such that the difference 
between $s^n$ and starting time of $\pi^n_{s_1}$ in $\tanh(\cdot)$ metric lies in 
$(\tilde{\epsilon}, 2\tilde{\epsilon})$. For $n$ large enough, we can always make such
a choice and only for this part of the argument we require $n$ to be large. Under the assumption 
$\Lambda^\epsilon_{[\tau_1 , \tau_2), n}\cap (\tau_1 , \tau_2) = \emptyset $ we have 
$\Xi^\epsilon_{[\tau_1 , \tau_2), n} = \emptyset $ as well 
and consequently  it follows that for all $s \in [\tau_1, 
\tau_2)$, the scaled path starting from $(y^n,s^n)$ at dynamic time $s$ does not have 
a member of $D^\epsilon_{[\tau_1 , \tau_2), n}$ on it's trajectory. 
Rest of the argument is exactly same as Lemma \ref{lem:PoissonClockSetIntersection}.      
\end{proof}

\begin{lemma}
\label{lem:DyDW_RCLL}
For each $n$, the scaled DyDW process $\{{\cal X}_n(\tau) : \tau \in [0,1]\}$
is RCLL a.s.
\end{lemma}
\begin{proof}
The DyDW process is constructed on the discrete set of scaled lattice points. Because of
Lemma \ref{lem:PoissonClockSetIntersection_1}, for each $n$ to 
show that the process $\{{\cal X}_n(\tau) : \tau \in [0,1]\}$ is RCLL, 
it is enough to show that the set $\Xi^\epsilon_n$ is contained in a compact box a.s.
Fix $n \in \N$. Properties of the metric space $\R^2_c$ 
ensures that  there exists  $M \in \N$ such that 
\begin{itemize}
\item[(i)] any scaled lattice point $(y^n, s^n)$ in the set $\Xi^\epsilon_n$ must have 
$ s^n \in [-M, M]$ and
\item[(ii)] at least one of the two outgoing paths starting from 
$(y^n, s^n) \in \Xi^\epsilon_n$ must intersect with the box $[-M, M]^2$.      
\end{itemize}
Since, the scaled DyDW paths start from scaled lattice points and 
  for any $\tau \in [0,1]$ the modulus of continuity of   
all the paths in ${\cal X}_n(\tau)$ is uniformly bounded by $\sqrt{n}$,
condition (i) and (ii) ensure that the set of points in $\Xi^\epsilon_n$ must 
be  contained in a compact box. This completes the proof. 
\end{proof}

In order to show that the sequence 
$\{ \overline{{\cal X}_n}(\tau) : \tau \in [0,1]\}_{n \in \N}$ is tight we need to show the 
following proposition. 
\begin{proposition}
\label{prop:RCLL_Tightness}
Fix $ \epsilon > 0$ and $\gamma \in (0,1)$. There exist $n_0 =  n_0(\epsilon, \gamma)\in \N$ 
and $\delta = \delta(\epsilon, \gamma)$ such that 
\begin{equation}
\label{eq:RCLL_Tight}
\P( \omega_{\overline{{\cal X}_n}}(\delta) > \epsilon \text{ for some }n \geq n_0) < \gamma.
\end{equation}
\end{proposition} 
We first describe a strategy to prove the above proposition. 
Because of Lemma \ref{lem:PoissonClockSetIntersection_1}, in order  to prove Proposition \ref{prop:RCLL_Tightness} it suffices to show that there exist
$n_0 =  n_0(\epsilon, \gamma)\in \N$ 
and $\delta = \delta(\epsilon, \gamma)$ such that with probability 
bigger than $1-\gamma$, for all $n \geq n_0$ there exists a partition 
(possibly random) $\{[\tau_i, \tau_{i+1}) : 0 \leq i \leq \ell -1 \}$ of $[0,1]$ with
\begin{itemize}
\item[(i)] $\min\{\tau_{i+1} - \tau_i : 0 \leq i \leq \ell -1\} > \delta$ and 
\item[(ii)] $\Lambda^\epsilon_{(\tau_i , \tau_{i+1}),n} = \emptyset$ for all 
 $ 0 \leq i \leq \ell -1 $.
\end{itemize}
Towards this we choose $m \in \N$ (we will specify the choice of $m$ later) 
and partition the interval $[0,1]$ into $m$ many  intervals of length $1/m$ 
given by $\{ [j/m, (j+1)/m) :0 \leq j \leq m -1 \}$. 
For simplicity of notation we denote $\Lambda^{\epsilon/2}_{[j/m, (j+1)/m),n}$ simply as 
$\Lambda^{\epsilon/2}_{j,n}$. We note that we have considered  $\epsilon/2$ instead of $\epsilon$
and the reason for considering $\epsilon/2$ will be explained shortly. 
For any set $A$ let $\# A$ denote the cardinality of $A$.
We define the event
\begin{align}
\label{def:Event_F}
F^m_n := ( \Lambda^{\epsilon/2}_{0,n} = \Lambda^{\epsilon/2}_{m-1,n} = \emptyset)  \bigcap \bigl( 
\cap_{j=1}^{m-2} ( \# ( \cup_{i = j-1}^{j+1}   \Lambda^{\epsilon/2}_{i,n} ) \leq 1 )  \bigr ). 
 \end{align}
In other words, the event $F^m_n$ ensures that 
for any $1 \leq j \leq m  - 2$ the set $ \Lambda^{\epsilon/2}_{j,n} $
can have at most one element. 
Further, if the set $ \Lambda^{\epsilon/2}_{j,n}$ is non-empty, 
then both the `neighbouring' sets $\Lambda^{\epsilon/2}_{(j-1),n}$ 
and $\Lambda^{\epsilon/2}_{(j+1),n}$ must be empty.

On the event $F^m_n$ for $n \geq n_0$ let $ I \subset \{0, 1, \cdots, m - 1\} $
be such that $\Lambda^{\epsilon/2}_{j,n}$ is non-empty if and only if $j \in I$
 and set 
\begin{align*}
\Lambda^{\epsilon/2}_{j,n} 
= \begin{cases}
\{ s_j \} & \text{ if }j \in I \\
 \emptyset & \text{ if }j \notin I.
\end{cases}
\end{align*}
Precisely, for each $j \in I$, the set $\Lambda^{\epsilon/2}_{j,n}$ is a singleton set
denoted by $\{ s_j \}$.
This allows us to choose a partition of the unit interval as 
$$
\Bigl ( \bigcup_{0\leq j \leq m-1 :
j - 1, j, j + 1 \notin I} [j/m , (j+1)/m)\Bigr )\bigcup 
\Bigl ( \bigcup_{j \in I} [(j-1)/m , s_j) \cup [s_j, (j+2)/m) \Bigr ).
$$ 
In other words, for $j \in I$ we modify the partition 
$\cup_{\ell = j-1}^{j+1}[\ell/m, (\ell+1)/m)$ as $[(j-1)/m, s_j) \cup [s_j, (j+2)/m)$. 
Choose $\delta < 1/m$ and we observe that  
each interval in the above partition has length strictly bigger than $\delta$. 
For $0 \leq j \leq m-1$ with $j - 1, j, j + 1 \notin I$ we have 
$\Lambda^{\epsilon/2}_{j,n} = \emptyset$ implying 
$\omega_{\overline{{\cal X}_n}}([j/m, (j+1)/m)) < \epsilon/2 < \epsilon$.
On the other hand for $j \in I$, a.s. $j/m$ is not a jump point and 
Remark \ref{rem:Subadditivity} ensures that we have 
\begin{equation}
\label{eq:Omega_Subadditve}
\omega_{\overline{{\cal X}}_n} ( [(j-1)/m , s_j)) \leq 
\omega_{\overline{{\cal X}}_n}\bigl ([(j-1)/m, j/m) \bigr ) + 
\omega_{\overline{{\cal X}}_n}\bigl ( [j/m, s_j) \bigr ).
\end{equation}
For both the intervals on the R.H.S., we have 
that the corresponding $\Lambda^{\epsilon/2}$ sets are empty giving us that   
 $ \omega_{\overline{{\cal X}}_n} ( [(j-1)/m , s_j))
 \leq \epsilon/2 + \epsilon/2 \leq \epsilon $.
  Similar logic applies for the part $\omega_{{\cal X}_n} ( [s_j, (j+2)/m ))$ too, 
as we have $\Lambda^{\epsilon/2}_{[s_j,(j+2)/m), n} \cap (s_j,(j+2)/m)  = \emptyset$. 
 This shows that on the event $F^{m}_n $ for any $0< \delta < 1/m$ we have 
 $\omega_{\overline{{\cal X}}_n}(\delta) < \epsilon $.
 Hence,  in order to prove Proposition \ref{prop:RCLL_Tightness}, 
it is enough to prove the following proposition. 
\begin{proposition}
\label{prop:RCLL_Tightness_1}
Fix $ \epsilon > 0$ and $\gamma \in (0,1)$. There exist $n_0  =  n_0(\epsilon, \gamma), 
m_0  =  m_0(\epsilon, \gamma)\in \N$ such that 
\begin{equation}
\label{eq:RCLL_Tight_1}
\P( F^{m_0}_n \text{ for all }n \geq n_0) \geq 1 - \gamma.
\end{equation}
\end{proposition} 
This proposition will be proved through a sequence of lemmas. 
We first describe our heuristics in words. 
 The exact value of $m_0 \in \N$  will be specified later (see (\ref{eq:m_0_n_3})).
We consider a partition $\{[j/m_0, (j+1)/m_0) : 0 \leq j \leq m_0 -1\}$ 
of the unit interval into sub intervals of width $1/m_0$.
Corresponding to the partition $\{[j/m_0, (j+1)/m_0) : 0 \leq j \leq m_0 - 1\}$ of 
$[0,1]$ we consider the set $\Lambda^{\epsilon/2}_{j} = \Lambda^{\epsilon/2}_{[j/m_0, (j+1)/m_0)}$
for $0 \leq j \leq m_0 - 1$. We choose $m_0 \in \N$ large enough so that 
for any $j$ the set $\Lambda^{\epsilon/2}_{j}$ contains at most one element and 
for most $j$'s, the set $\Lambda^{\epsilon/2}_{j}$ is empty. Precisely
we define an event $B_{m_0}$ such that 
\begin{align*}
B_{m_0} :=  & ( \Lambda^{\epsilon/2}_{0} = \Lambda^{\epsilon/2}_{m_0 - 1}  = \emptyset ) 
\bigcap \bigl ( \#(\Lambda^{\epsilon/2}_{j} \cup \Lambda^{\epsilon/2}_{j+1}) 
\leq 1 \text{ for all }1 \leq j \leq m_0 - 2 \bigr ) ,
\end{align*}
and choose $m_0$ large enough so that $\P(B_{m_0})$ is close to one. 
Here is a brief justification about why such a choice of $m_0$ is always possible. 
As the set $\Lambda^{\epsilon/2}$ is a.s. finite, the r.v. $Z 
:= \min\{ s_1, |s_1 - s_2|, 1 - s_2 : s_1 , s_2 \in \Lambda^{\epsilon/2} \}$
is strictly positive with probability $1$. We observe that on the event $\{ Z > \delta\}$, 
the event $B_{m_0}$ holds for any $m_0 \in \N$ with $m_0 < \delta/3$. As $\P( Z \leq \delta )$ 
decreases to zero as $\delta \downarrow 0$, we can choose $m_0$ large enough 
to make $\P(B_{m_0})$ arbitrarily close to one.

Based on the above choice of $m_0$ we show that with high probability for all large $n$
the set $\Lambda^{\epsilon/2}_{j,n}$ is non-empty only if the set 
$\Lambda^{\epsilon/2}_{j}$ is non-empty and a non-empty $\Lambda^{\epsilon/2}_{j,n}$
contains exactly one element. This implies (\ref{eq:RCLL_Tight_1}). 
In other words, to prove (\ref{eq:RCLL_Tight_1}) it suffices to show that there 
exist $n_0 = n_0(\epsilon, \gamma), m_0 = m_0(\epsilon, \gamma) \in \N$ such that   
\begin{equation}
\label{eq:Fluctuation_X_n_1}
 \P \bigl( B_{m_0} \bigcap ( \cap_{j=0}^{m_0 - 1}(\# \Lambda^{\epsilon/2}_{j,n} \leq 
 \# \Lambda^{\epsilon/2}_{j}  \text{ for all }n \geq n_0 ) ) \bigr) \geq 1 - \gamma.
\end{equation}
In the following, through a sequence of lemmas we prove (\ref{eq:Fluctuation_X_n_1}) 
and thereby prove Proposition \ref{prop:RCLL_Tightness_1} to obtain tightness.
Towards this we work on a specific probability space to use properties of almost
sure convergence. Below we will construct our required probability space of 
almost sure convergence.

We consider a partition $\{[j/m_0, (j+1)/m_0) : 0 \leq j \leq m_0 -1\}$ 
of the unit interval into sub intervals of width $1/m_0$.
Corresponding to this partition of the dynamic time interval, 
the $m_0 + 1$ dimensional vector of `corresponding' nets is given by
$$
(\NN, \NN_0, \cdots , \NN_{m_0 - 1}).
$$ 
We observe that  for any $0 \leq j \leq m_0 - 1$ we have 
$ \NN_j \stackrel{d}{=} \NN(1/m_0)$.
Further, we observe that 
$$
\pi \in \WW(s) \text{ for some }s \in [j/m, (j+1)/m ) \text{ only if }
\pi \in \NN_j.
$$ 
The next proposition proves that as $n \to \infty$, the vector of scaled discrete nets 
$\bigl( \NN_n, \NN_{0,n} , \cdots, \NN_{m_0 - 1,n}\bigr)$ as ${\cal H}^{m_0+1}$ 
valued random variable converges in distribution to the vector of corresponding 
 nets $\bigl( \NN , \NN_0, \cdots, \NN_{m_0 - 1}\bigr)$.     
\begin{proposition}
\label{prop:JtConv_Nets}
As $n \to \infty$, we have 
$$
\bigl( \NN_n, \NN_{0,n}, \cdots, \NN_{m_0 -1, n}) \Rightarrow 
\bigl( \NN , \NN_0, \cdots, \NN_{m_0 - 1}\bigr),
$$
as ${\cal H}^{m_0 +1}$ valued random variables.
\end{proposition}
We postpone proving Proposition \ref{prop:JtConv_Nets} for the moment and proceed. 
In fact, as the DyDW paths are independent till the time they meet, Sun et. al. \cite{SS08}
showed that the scaled rightmost and leftmost paths jointly with their 
first meeting times converge to left right Brownian paths and their first meeting 
times. The same holds for  each of the dynamic time interval $[j/m_0, (j+1)/m_0)$
for $0 \leq j \leq m_0 - 1$. We need to introduce some notations. 
For $(x,t) \in \Q^2$ let $l^{(x,t)}$ and $r^{(x,t)}$ respectively 
denote the leftmost and rightmost path in the Brownian net $\NN$ starting from $(x,t)$. 
Let $l^{(x,t)}_{j}$ and $r^{(x,t)}_{j}$ respectively denote the same for
the Brownian net $\NN_j$ constructed over dynamic time interval $[j/m_0, (j+1)/m_0)$.  
For $(x,t), (y,s) \in \Q^2$ let $\tau^{(x,t),(y,s)}$ denote the first meeting time 
of $l^{(x,t)}$ and $r^{(y,s)}$ and let $\tau^{(x,t),(y,s)}_j$ denote the same 
for paths $l^{(x,t)}_j$ and $r^{(y,s)}_j$.   
For $(x,t) \in \Q^2$ let $(x^n, t^n) \in \Z^2_{\text{even}}$ be such that 
$S_n(x^n, t^n) \to (x,t)$ as $n \to \infty$. 
For $(x,t) \in \Q^2$ let $l^{(x^n,t^n)}_n$ and $r^{(x^n,t^n)}_n$ respectively 
denote the scaled leftmost and rightmost DW path in $\NN_n$ starting from $S_n(x^n,t^n)$.
$l^{(x^n,t^n)}_{j,n}$ and $r^{(x^n,t^n)}_{j,n}$ denote the same in $\NN_{j,n}$.
For $(x,t), (y,s) \in \Q^2$ the first meeting times $\tau^{(x,t),(y,s)}_n$ and 
$\tau^{(x,t),(y,s)}_{j,n}$ are similarly defined. Assuming Proposition \ref{prop:JtConv_Nets}, we have  
\begin{align}
\label{eq:JtConv_MeetingTime_BNetVector}
& ( l^{(x^n,t^n)}_{n}, r^{(y^n,s^n)}_{n}, \tau^{(x,t),(y,s)}_{n}, l^{(x^n,t^n)}_{0,n},r^{(y^n,s^n)}_{0,n}, \tau^{(x,t),(y,s)}_{0,n}, 
\cdots ,l^{(x^n,t^n)}_{m_0-1,n},r^{(y^n,s^n)}_{m_0-1,n}, \tau^{(x,t),(y,s)}_{m_0 -1,n} \bigr )_{(x,t),(y,s) \in \Q^2} \nonumber \\
& \Rightarrow 
( l^{(x,t)}, r^{(y,s)}, \tau^{(x,t),(y,s)}, l^{(x,t)}_{0},r^{(y,s)}_{0}, 
\tau^{(x,t),(y,s)}_{0}, \cdots , l^{(x,t)}_{m_0-1},r^{(y,s)}_{m_0-1}, 
\tau^{(x,t),(y,s)}_{m_0 -1} \bigr )_{(x,t),(y,s) \in \Q^2}.
\end{align}
Here we give a brief justification of Equation (\ref{eq:JtConv_MeetingTime_BNetVector}).
From Proposition \ref{prop:JtConv_Nets} we have joint convergence of 
the vector of scaled discrete nets to the vector of Brownian nets. This 
guarantees that the vector of left right scaled DyDW paths in (\ref{eq:JtConv_MeetingTime_BNetVector})
converges to the vector of respective left right coalescing Brownian motions. 
Since, scaled DyDW paths are independent till the time they meet (at a fixed dynamic time point), respective first intersection times jointly converge too.

By Skorohod's embedding, we assume that we are working on a probability space so 
that the convergence as in (\ref{eq:JtConv_MeetingTime_BNetVector}) happens almost surely. 
In what follows, we will continue to work on this 
probability space and we will use consequences of this almost sure convergence extensively. 
It should be mentioned here that the exact choice of $m_0 \in \N$ will be mentioned later 
in Equation (\ref{eq:m_0_n_3}). 

The next lemma follows from consequences of almost sure convergence in path space.
\begin{lemma}
\label{lem:Lambda_Lambda^n_set_comparison}
For any $0 \leq j \leq m_0 - 1$, we have equality of the following events 
\begin{equation*}
\{\Lambda^{\epsilon/2}_j \neq \emptyset \} = 
\{\Lambda^{\epsilon/2}_{j, n} \neq \emptyset \text{ for infinitely many }n\}.
\end{equation*}
\end{lemma}
\begin{proof} Because of stationary nature of 
our model, it suffices to prove Lemma \ref{lem:Lambda_Lambda^n_set_comparison} 
for $j = 0$.  
On the event  $\{\Lambda^{\epsilon/2}_0 \neq \emptyset \}$, 
we must have $\{ \Xi^{\epsilon/2}_0 \neq \emptyset \}$ as 
the set $\Lambda^{\epsilon/2}_0$
consists of Poisson clock rings in the dynamic time interval $[0,1/m_0)$
associated to points in the set $\Xi^{\epsilon/2}_0$. 
On the other hand, it is not difficult to see that we must have 
$$
\{ \Xi^{\epsilon/2}_0 \neq \emptyset \} \subseteq
\{ \Lambda^{\epsilon/2}_0 \neq \emptyset \},
$$
as there must be some switching event in the interval $[0, 1/m_0)$ (actually $(0, 1/m_0)$) in order to have the set 
$ \Xi^{\epsilon/2}_0$ non-empty.
Hence, the two events are equal. 
Similar reasoning holds for discrete scaled nets as well.
Hence, to prove Lemma \ref{lem:Lambda_Lambda^n_set_comparison} it is enough to show equality
of the following events 
$$
\{ \Xi^{\epsilon/2}_{0, n} \neq \emptyset \text{ for infinitely many }n\}
= \{ \Xi^{\epsilon/2}_0 \neq \emptyset \} .
$$
On the event $\{ \Xi^{\epsilon/2}_{0, n} \neq \emptyset 
\text{ for infinitely many }n\text{'s}\}$, almost sure convergence of path families ensures that the limiting Brownian net $\NN_0$ (constructed over dynamic time interval $[0,1/m_0)$)
must have a {\it separation point} $(x,t)$ with an incoming 
forward path of age at least $\tilde{\epsilon} = \tilde{\epsilon}(\epsilon/2)$.
 Further,  diameter of the maximal excursion set generated at $(x,t)$ is bigger than $\epsilon$. 
 This implies that the set $\Xi^{\epsilon/2}_{[0, 1/m_0)}$ must be non-empty. 
  Hence, we have 
 $$
 \{ \Xi^{\epsilon/2}_{0, n} \neq \emptyset \text{ for infinitely many }n\}
\subseteq \{ \Xi^{\epsilon/2}_0 \neq \emptyset \}.
 $$
 On the other hand, on the event $\{ \Xi^{\epsilon/2}_0 \neq \emptyset \}$
 consider a separating point $(x,t) \in \Xi^{\epsilon/2}_0$. 
 There must exist skeletal rightmost and leftmost Brownian paths $\pi^r, \pi^l \in \NN$
  passing through $(x,t)$ and forming the right boundary and left boundary 
  of the maximal excursion set generated at $(x,t)$ respectively. 
  Further, their first meeting time must be smaller than $t $. 
Almost sure convergence of rightmost and leftmost path families 
together with their first meeting times ensure that there 
must be sequences of paths $\{\pi^n_l : n \in \N\}$ and 
 $\{\pi^n_r : n \in \N\}$ approximating paths $\pi^r$ and $\pi^l$  respectively
 and their first meeting time converges to that of $\pi^r$ and $\pi^l$ which 
 is smaller than $t$. 
 
After intersection, scaled discrete web paths (rightmost and leftmost) 
continue to move together till they encounter a separating point 
(for path family $\NN_{0,n}$). As $\{\pi^n_l : n \in \N\}$ and  
$\{\pi^n_r : n \in \N\}$ approximate paths $\pi^r$ and $\pi^l$ 
with an $\epsilon$ fat maximal excursion set in between them, after the first meeting time
 there must be a separating point 
$(x_n, t_n) \in S_n(\Z^2_{\text{even}})$ on their trajectory with maximal excursion set 
of diameter $\epsilon$ at least. More precisely, 
for infinitely many $n$'s there exist $(x_n, t_n)$  such that
$$
\pi^n_r(t_n) = \pi^n_l(t_n) = x_n \text{ and }(x_n, t_n) \in D^{\epsilon/2}_{0,n}.
$$ 
 To complete the proof we need to ensure that $(x_n, t_n) \in 
 A^{\tilde{\epsilon}/2}_{0,n}$ as well. 
 This follows from the fact that the limiting net has incoming path(s) at $x,t$
 of age $\tilde{\epsilon}(\epsilon/2)$ at least. For large $n$, there 
 must be approximating incoming path(s) in ${\cal N}_{0,n}$ passing through $(x_n, t_n)$. 
This completes the proof. 
\end{proof}
Lemma \ref{lem:Lambda_Lambda^n_set_comparison} suggests that for all
$0 \leq j \leq m_0 - 1$, on the event $\{\Lambda^{\epsilon/2}_j = \emptyset \}$, 
we must have that the r.v. $\sup\{ n \in \N : \Lambda^{\epsilon/2}_{j, n} \neq \emptyset \}$ 
is a.s. finite. This observation is used to prove the next lemma. 
\begin{lemma}
\label{lem:Xi_n_Fluctuation_bd}
For each $\gamma \in (0,1)$, there exist $m_0 = m_0(\epsilon,\gamma),
n^\prime  = n^\prime(\epsilon, \gamma)  \in \N$  such that 
\begin{equation}
\P \bigl( \Lambda^{\epsilon/2}_{[0,1/m_0), n} \neq \emptyset 
\text{ for some }n \geq n^\prime \bigr ) < \gamma.
\end{equation}
\end{lemma}
\begin{proof} 
We note that the random set  $\Lambda^{\epsilon/2}$ is a.s. finite. 
Further as $m \to \infty$, the set  $\Lambda^{\epsilon/2}_{[0,1/m)}$ 
{\it monotonically decreases}
to empty set. Hence, we can choose $m_0 \in \N$ such that 
\begin{equation*}
\P( \Lambda^{\epsilon/2}_{[0,1/m_0)} \neq \emptyset ) < \gamma/2,
\end{equation*}
where $\Lambda^{\epsilon/2}_{[0,1/m_0)}$ is defined 
considering the DyBW process over dynamic time interval $[0,1/m_0)$. 
 We define $X$ as the non-negative integer valued random variable
\begin{align*}
X := \sup\{ n \in \N : \Lambda^{\epsilon/2}_{[0,1/m_0),n} \neq \emptyset \}.
\end{align*} 
By the previous lemma  on the event $\{ X = \infty \}$, the set 
$\Lambda^{\epsilon/2}_{[0, 1/m_0)}$ must be non-empty. This gives us that
\begin{align*}
\P(X = \infty )  \leq \P( \Lambda^{\epsilon/2}_{[0,1/m_0)} \neq \emptyset ) < \gamma/2.
\end{align*}
Hence, we can choose $n^\prime \in \N$ such that 
 $\P( X \leq n^\prime ) > 1 - \gamma$. 
This completes the proof.
\end{proof}
 Now we are ready to prove Proposition \ref{prop:RCLL_Tightness_1}. 

\noindent {\bf Proof of Proposition \ref{prop:RCLL_Tightness_1}:} 
As discussed earlier, it suffices to prove  Equation (\ref{eq:Fluctuation_X_n_1}).
Fix $\epsilon > 0$ and $\gamma \in (0, 1)$. Remember that we need to find $ m_0  = 
m_0(\epsilon, \gamma)> 0 $ and $n_0  = n_0(\epsilon, \gamma)\in \N$ such that 
Equation (\ref{eq:Fluctuation_X_n_1}) holds. 

For $\delta > 0$ we define the event $A(\delta)$ as 
\begin{align}
\label{def:Event_A_1}
A(\delta) &:= \bigl\{ \min\{|s_1 -s_2|\wedge s_1\wedge (1 - s_2) : 
s_1, s_2 \in \Lambda^{\epsilon/2}\} > \delta \bigr\}.
\end{align}
We first choose $m_1 = m_1(\epsilon, \gamma)\in \N$ such that 
$\P( A(2/m_1)^c ) < \gamma/8$.
Recall that the random set $\Lambda^{\epsilon/2}$ consists of 
finitely many distinct points a.s. and hence, such a choice of $m_1$ is always possible.
Further, on the event $A(2/m_1)$ we must have 
$$
\# \Lambda^{\epsilon/2} \leq \lfloor m_1/2\rfloor + 1 < m_1 . 
$$  
Next, using Lemma \ref{lem:Xi_n_Fluctuation_bd} we choose $m_0 \geq m_1$ and 
$n_3 = n^\prime(\epsilon/2 , \gamma/(8 m_1))$ 
such that   
\begin{align}
\label{eq:m_0_n_3}
\P\bigl( \# \Lambda^{\epsilon/2}_{0,n}
= \# \Lambda^{\epsilon/2}_{[0,1/m_0),n} \geq 1 
\text{ for some }n \geq n_3 \bigr ) < \gamma/(8 m_1).
\end{align}
As $m_0 \geq m_1$, we observe that on the event $A(2/m_1)$, any 
interval of the form $[j/m_0, (j+1)/m_0)$ can have at most one element from the set 
$\Lambda^{\epsilon/2}$ and for any such interval, the neighbouring interval(s) cannot have any element from the set $\Lambda^{\epsilon/2}$.

For $0 \leq j \leq m_0 -1$, let $N^j = N^j(\epsilon, m_0)$ denote
 a non-negative integer valued random variable defined as
 \begin{align*}
 N^j := \mathbf{1}_{\{\Lambda^{\epsilon/2}_{j} =  \emptyset\}} 
 \sup\{ n \in \N : \Lambda^{\epsilon/2}_{j,n} \neq \emptyset \}.
\end{align*}  
Lemma \ref{lem:Lambda_Lambda^n_set_comparison} ensures that 
the non-negative r.v. $N^j$ is a.s. finite. 
Using the collection $\{N^j : 0 \leq j \leq m_0 -1\}$ we define 
\begin{align*}
N = N(\epsilon, m_0) := \max\{N^j : 0 \leq j \leq m_0 - 1\}, 
\end{align*}
which is an a.s. finite non-negative integer valued r.v. as well.
 We choose $n_4 \geq n_3  $ so that 
\begin{align}
\label{eq:m_0}
\P( N \geq n_4) < \gamma/8.
\end{align}
We will show that with the above choice of $m_0$ and with $n_0 = n_4$
Equation (\ref{eq:RCLL_Tight_1}) holds. We define the event $E$ as  
\begin{align*}
E := \bigcap_{n \geq n_0}\{ \# \Lambda^{\epsilon/2}_{j,n} \leq \# 
\Lambda^{\epsilon/2}_{j} \text{ for all } 0 \leq j \leq m_0 - 1 \}.
\end{align*} 
With the choice of $m_0$ and $n_0$ are as before, we observe the following inclusion of events:
$$
\{ F^{m_0}_n \text{ for all }n \geq n_0\} \subseteq A(2/m_1)\cap E.
$$
Hence, in order to prove (\ref{eq:RCLL_Tight_1}), it is enough to show that 
\begin{align}
\label{eq:ComplementEvent_Ineq}
2\P( A(2/m_1)^c) + \P(N \geq n_0) + \P(E^c \cap  (N < n_0)\cap A(2/m_1)) < \gamma.
\end{align}
The choice of $m_1$ and $n_0$ ensure that it suffices to show that
$$
\P(E^c \cap  (N < n_0)\cap A(2/m_1)) < \gamma/2.
$$   

Let $\Gamma$ denote the collection of all possible subsets of the set 
$\{1, \cdots, m_0-2\}$ with {\it at most} $m_1$ many elements such that 
these subsets do not have consecutive elements. 
In other words, 
\begin{align*}
\Gamma := \bigl \{ & \lambda\subseteq \{1, \cdots, m_0 - 2\} : \# \lambda \leq m_1 
\text{ and }\nexists j \in \{1, \cdots, m_0 - 2\} \text{ such that both }\\
 & j \text{ and } j+1 \text{ are in } \lambda \bigr \}.
\end{align*}
For $\lambda \in \Gamma$ we define the event $G_\lambda$ as 
\begin{align*}
G_\lambda := & \bigl( \cap_{j \in \lambda}( \Lambda^{\epsilon/2}_{j,n}  \neq \emptyset 
\text{ for infinitely many }n) \bigr)
 \bigcap \bigl( \cap_{j \notin \lambda}( \Lambda^{\epsilon/2}_{j,n} = \emptyset 
 \text{ for all }n \geq n_0 ) \bigr).
\end{align*}
It is not difficult to see that for $\lambda, \lambda^\prime \in \Lambda $
with $\lambda \neq \lambda^\prime$, the events $G_\lambda$ and $G_{\lambda^\prime}$ are {\it disjoint}. 
By definition, on the event $A(2/m_1)$ we have $\Lambda^{\epsilon/2}_0
 = \Lambda^{\epsilon/2}_{m_0 - 1} = \emptyset$.  
Hence, Lemma \ref{lem:Lambda_Lambda^n_set_comparison} gives us the following inclusion relation
$$
\bigl(  A(2/m_1) \cap (N < n_0) \bigr ) \subseteq \cup_{\lambda \in \Gamma} G_\lambda.
$$
Therefore, we can express $\P\bigl( E^c \cap (N < n_0) \cap A(2/m_1) \bigr)$ as 
\begin{align}
\label{eq:Tightness_1}
& \P(E^c \cap (N < n_0) \cap A(2/m_1))\nonumber\\
& \leq  \P \bigl( E^c \cap (\cup_{\lambda \in \Lambda} G_\lambda) \bigr )\nonumber\\
& =  \P \bigl( \cup_{\lambda \in \Lambda} (E^c \cap  G_\lambda) \bigr )\nonumber\\
& = \sum_{\lambda \in \Gamma} \P \bigl( E^c \mid  G_\lambda \bigr )\P(G_\lambda)\nonumber\\
& \leq \sum_{\lambda \in \Gamma} \P \bigl( \cup_{ j \in \lambda }  
\bigl( \# \Lambda^{\epsilon/2}_{j,n} \geq 2 \text{ for some }n \geq n_0 \bigr ) \mid  G_\lambda \bigr )
\P(G_\lambda)
\end{align}   
The last step follows from the fact that given that the event 
$G_\lambda$ has occurred for some $\lambda \in \Lambda$,  
the only way that the complement event $E^c$ can occur is due to 
the set $\Lambda^{\epsilon/2}_{j,n}$ for some  $j \in \lambda$ and for some $ n \geq n_0$.
Further, Lemma \ref{lem:Lambda_Lambda^n_set_comparison} ensures that 
on the event $G_\lambda$, the set $\Lambda^{\epsilon/2}_{j}$ must be non-empty.
Hence, for any $\lambda \in \Lambda$ conditional that the event $G_\lambda$ has occurred, 
we have the following event inclusion
$$
E^c \subseteq \bigl( \# \Lambda^{\epsilon/2}_{j,n} \geq 2 \text{ for some }n \geq n_0 \bigr ) . 
$$
This justifies the last inequality in (\ref{eq:Tightness_1}).
We  can now write (\ref{eq:Tightness_1}) as
\begin{align}
\label{eq:Tightness_2}
& \sum_{\lambda \in \Gamma} \P \Bigl( \cup_{ j \in \lambda }  
\bigl( \# \Lambda^{\epsilon/2}_{j,n} \geq 2 \text{ for some }n \geq n_0 \bigr ) \mid  G_\lambda \Bigr )
\P(G_\lambda)\nonumber\\
\leq  & \sum_{\lambda \in \Gamma} m_1 \P\Bigl ( \bigl ( \# \Lambda^{\epsilon/2}_{j,n} 
 \geq 2 \text{ for some }n \geq n_0 \bigr )
\mid G_\lambda \Bigr ) \P(G_\lambda). 
\end{align}
The last inequality follows from application of union bound and 
the fact that $\#\lambda \leq m_1$ for all $\lambda \in \Lambda$.
We observe that the collection of random variables  $\{\Lambda^{\epsilon/2}_{j,n} : n \in 
\N \}$ depends on the evolution of the i.i.d. Markov processes $\{I^n_{(x,t)}(\cdot) : (x,t)
\in \Z^2_{\text{even}}\}_{n \in \N}$ over dynamic time domain $[j/m_0, (j+1)/m_0)$.

Fix $n \in \N$. For $0 \leq j \leq m_0 - 1$ let $X^j_n$ denote the 
non-negative integer valued random variable defined as $X^j_n := \# \Lambda^{\epsilon/2}_{j,n}$.
For $0 \leq j \leq m_0 - 1$ we define 
$$
Z^j := (X^j_1, X^j_2, \cdots ) 
$$ 
which takes vales in $(\N \cup \{0\})^\N$. 
The next lemma shows that the set $\{ Z^j : 0 \leq j \leq m_0 - 1\}$
forms an i.i.d. collection of random variables taking vales in $(\N \cup \{0\})^\N$.   
\begin{lemma}
\label{lemma:iid_1}
$\{ Z^j : 0 \leq j \leq m_0 - 1\}$ forms an i.i.d. collection of random variables.
\end{lemma}
\begin{proof}
Because of stationarity of our process, it follows that for fixed $n$
the random variables $X^j_n$'s for $0 \leq j \leq m_0 - 1$ are identically distributed.
Fix Borel subsets $B^0, . . . , B^{m_0 - 1}$ in appropriate space. 
Let $I_j(B^j)$ be the indicator random variable of the event 
$\{ Z^j \in B^j\}$. 
For $0 \leq j \leq m_0 - 1$, set the $\sigma$-field 
$$
{\cal G}_{j} := \sigma \bigl ( (N^n_{(x,t)}(\tau), I^\ell_{(x,t)}) : (x,t) \in 
\Z^2_{\text{even}}, \tau \in [0,j/m_0), n \geq 1 , 0 \leq \ell \leq N^n_{(x,t)}(j/m_0)\bigr ).
$$
Now we have
\begin{align*}
\P( Z^j \in B^j \text{ for } 0 \leq j \leq m_0 - 1)
& = \E(\Pi_{j=0}^{m_0-1} I_j(B^j))\\
& = \E \Bigl(\E(\Pi_{j=0}^{m_0-1} I_j(B^j) \mid {\cal G}_{m_0 - 1})\Bigr)\\
& = \E \Bigl( \Pi_{j=0}^{m_0-2} I_j(B^j) \E( I_{m_0 - 1}(B^{m_0 - 1}) \mid {\cal G}_{m_0 - 1})\Bigr)\\
& = \E \Bigl( \Pi_{j=0}^{m_0-2} I_j(B^j) \Bigr) \P( Z^{0} \in B^{m_0 - 1}),
\end{align*} 
where the penultimate step follows from the fact that the r.v. $I_j(B^j)$ is measurable 
w.r.t. the $\sigma$-field ${\cal G}^{m_0 - 1}$. The last equality follows from stationarity 
nature of our model which ensures equality of distribution of $Z^j$'s.
Applying the same argument repetitively completes the proof. 
\end{proof}
Lemma \ref{lemma:iid_1} gives us that the event $( \# \Lambda^{\epsilon/2}_{j,n} 
 \geq 2 \text{ for some }n \geq n_0 )$ does not depend on the collection
 $\{ Z^i : 0 \leq i \leq m_0 - 1, i \neq j \}$. Note that the event 
 $(\# \Lambda^{\epsilon/2}_{i,n} \geq 1\text{ for infinitely }n )$ is measurable 
 w.r.t. the $\sigma$-field $\sigma(Z^i)$. 
Hence, Equation (\ref{eq:Tightness_2}) becomes
\begin{align}
\label{eq:Tightness_3}
& \sum_{\lambda \in \Gamma} m_1 \P\Bigl ( \bigl ( \# \Lambda^{\epsilon/2}_{j,n} 
 \geq 2 \text{ for some }n \geq n_0 \bigr )
\mid G_\lambda \Bigr ) \P(G_\lambda) \nonumber \\
\leq & \sum_{\lambda \in \Gamma}m_1 \P\Bigl ( \bigl (  
   \# \Lambda^{\epsilon/2}_{j,n} \geq 2 \text{ for some }n \geq n_0 \bigr ) 
   \mid ( \# \Lambda^{\epsilon/2}_{j,n} \geq 1\text{ for infinitely }n ) \Bigr ) \P(G_\lambda)\nonumber\\
   \leq & \sum_{\lambda \in \Gamma}m_1 \P\Bigl ( \bigl (  
   \# \Lambda^{\epsilon/2}_{0,n} \geq 2 \text{ for some }n \geq n_0 \bigr ) 
   \mid ( \# \Lambda^{\epsilon/2}_{0,n} \geq 1\text{ for infinitely }n ) \Bigr ) 
   \P(G_\lambda) .
\end{align}
The last step follows from stationarity of our model.
We recall that for any $s>0$ the set $\Lambda^{\epsilon/2}_{[0,s],n}$ is 
defined based on the scaled DyDW (dynamic discrete web) 
restricted over dynamic time interval $[0,s]$ only.  
We define, 
\begin{align*}
 \sigma^n := (\inf\{ s \in [0, 1/m_0) : \# \Lambda^{\epsilon/2}_{[0,s],n}
  \geq 1 \})\wedge 1/m_0.
\end{align*}
We observe that for every $n$, the r.v. $\sigma^n$ is a stopping time w.r.t. the filtration 
$$
\Bigl \{ {\cal G}^s_{0,n} := \sigma \bigl ( (N^n_{(x,t)}(\tau), I^\ell_{(x,t)}) : (x,t) \in 
\Z^2_{\text{even}}, \tau \in [0,s],  0 \leq \ell \leq N^n_{(x,t)}(s)\bigr ) : s \in [0, 1/m_0) \Bigr \}.
$$
The strong Markov property of the scaled process (w.r.t. dynamic time)
allows us to obtain 
\begin{align}
\label{eq:Tightness_4}
& \P \bigl ( (\# \Lambda^{\epsilon/2}_{0,n} \geq 2 \text{ for some }n \geq n_0)
 \mid (\# \Lambda^{\epsilon/2}_{0,n} \geq 1 \text{ for infinitely many }n) \bigr )\nonumber\\
 = & \P \bigl ( (\# \Lambda^{\epsilon/2}_{(\sigma^n, 1/m_0),n} \geq 1 \text{ for some }
 n \geq n_0)   \mid (\sigma^n  < 1/m_0 \text{ for infinitely many }n) \bigr )\nonumber\\
 = & 
 \P \Bigl (  \bigcup_{n \geq n_0} \bigl( \# \Lambda^{\epsilon/2}_{(0,1/m_0 - \sigma^n),n}
  \geq 1 \mid {\cal G}^{\sigma^n}_{0,n} \bigr ) \Bigr )
  \end{align} 
As $\sigma^n$ is a stopping time for the scaled DyDW process, applying Lemma \ref{lem:Xi_n_Fluctuation_bd} to (\ref{eq:Tightness_4}) we obtain
$$
\P \Bigl (  \bigcup_{n \geq n_0} \bigl( \# \Lambda^{\epsilon/2}_{(0,1/m_0 - \sigma^n),n}
  \geq 1 \mid {\cal G}^{\sigma^n}_{0,n} \bigr ) \Bigr )
  \leq \P \bigl ( \# \Lambda^{\epsilon/2}_{[0,1/m_0 )}
  \geq 1 \text{ for some }n \geq n_0 \bigr )  < \gamma/(8 m_1).
$$
Finally, putting this estimate in (\ref{eq:Tightness_2}) gives us that 
$$
\P(E^c \cap (N < n_0)\cap A(2/m_1)) < \gamma/8 
\sum_{\lambda \in \Gamma}\P(G_\lambda ) = 
\gamma/8 \P( \cup_{\lambda \in \Gamma} G_\lambda ) \leq \gamma/8.
$$ 
As discussed earlier, this completes the proof of (\ref{eq:RCLL_Tight_1}) and 
thereby proves Proposition \ref{prop:RCLL_Tightness_1}. 
\qed

In order to complete our proof we need to prove Proposition \ref{prop:JtConv_Nets}.  
To prove Proposition \ref{prop:JtConv_Nets} we need 
the following notion of hopping at intersection times which is taken from \cite{SS08}.

For any collection of paths ${\cal K} \subset \Pi$ we let ${\cal H}_{\text{int}}({\cal K})$ 
denote the smallest set of paths containing ${\cal K}$ that is 
closed under hopping at intersection times, that is,
${\cal H}_{\text{int}}({\cal K})$ is the set of all paths $\pi \in \Pi$ of the form
\begin{equation}
\label{def:Hopping}
\pi = \cup^m_{i = 1} \{(\pi_i(t), t) : t \in [t_{k-1}, t_k]\},
\end{equation}
where $\pi_1, \cdots , \pi_m \in {\cal K}$, $\sigma_{\pi_1}
= t_0 < \cdots < t_m = \infty $, and $t_i$ is an intersection time
of $\pi_i$ and $\pi_{i+1}$ for each $i = 1, \cdots , m - 1$.

The following corollary gives us a way to identify the corresponding Brownian net corresponding 
to the DyBW $\{\WW(\tau) : \tau \in [0,1]\}$.
\begin{corollary}
\label{cor:Correspond_Net}
Let $\{ \WW(\tau) : \tau \in [0,1] \}$ denote the DyBW process with $\NN$ denoting 
the corresponding net. Let ${\cal K} = {\cal K}(\WW)$ denote the collection of paths
\begin{align}
\label{def:DyBW_AllPaths}
{\cal K} := \{ \pi : \pi \in \WW(\tau) \text{ for some } \tau \in [0,1]\cap \Q\} \subseteq \Pi,
\end{align} 
 consisting of DyBW paths at rational dynamic time points.
 Let $\NN^\prime$ be a standard Brownian net defined on the same probability 
 space such that ${\cal K}(\WW) \subset \NN^\prime$ a.s.
Then we must have $\NN^\prime = \NN$ a.s., i.e., $\NN^\prime$ must be the corresponding net only.  
\end{corollary}
\begin{proof} 
We proved that the DyBW has ${\cal H}$-valued RCLL paths a.s. 
Hence, $\overline{{\cal K}}$ contains the set
$\{\pi : \pi \in \WW(\tau) \text{ for some } \tau \in [0,1] \}$. 
As $\NN^\prime$ gives a compact collection collection of paths, our assumption 
$$
{ \cal K } \subseteq \NN^\prime \text{ implies }
\overline{{\cal K}} \subseteq \NN^\prime.
$$ 
This implies that  
\begin{equation}
\label{eq:Net_DyBW}
\{ \pi : \pi \in \WW(\tau) \text{ for some } \tau \in [0,1] \} \subseteq \NN^\prime.
\end{equation}
Next for any $n \in \N$, 
paths in the `corresponding' partial net $\NN_n$ belong to 
${\cal H}_{\text{int}}(\{ \WW(\tau) : \tau \in [0,1] \})$. 
By definition for any ${\cal K}_1 \subseteq {\cal K}_2 \subseteq \Pi$ we have
${\cal H}_{\text{int}}({\cal K}_1) \subseteq {\cal H}_{\text{int}}({\cal K}_2)$.
Since, Brownian net is closed under hopping (see Proposition 1.4. of \cite{SS08}), 
for every $n \in \N$ Equation (\ref{eq:Net_DyBW}) allows us to obtain 
$$
\NN_n \subseteq {\cal H}_{\text{int}}(\{ \WW(\tau) : \tau \in [0,1] \}) \subseteq 
{\cal H}_{\text{int}}(\NN^\prime) = \NN^\prime.
$$ 
On the other hand, 
\cite{NRS10} showed (see Theorem of \cite{NRS10}) for the corresponding net $\NN$ that
$$
\NN = \lim_{n \to \infty} \NN_n \text{ a.s}.
$$ 
As $\NN^\prime$ is a compact collection of paths a.s., 
this implies that $\NN \subseteq \NN^\prime$ as well.
Since they both have the same distribution, $\NN^\prime$ can not have more paths
and we must have $\NN^\prime = \NN$ a.s.
This completes the proof. 
\end{proof}

Next we present an extension of the above corollary. Recall the 
$m_0 + 1$ dimensional random vector
of the ``corresponding" Brownian nets $(\NN, \NN_0, \NN_1, \cdots , \NN_{m_0-1})$ 
where $\NN_j$ is constructed using markings in the dynamic time 
interval $[j/m_0, (j+1)/m_0)$ only.  
\begin{corollary}
\label{cor:Correspond_NetVector}
Let $(\NN^\prime, \NN^\prime_0, \NN^\prime_1, \cdots , \NN^\prime_{m_0-1})$ 
denote a $(m_0 + 1)$ dimensional 
random vector of Brownian nets such that $\NN^\prime \stackrel{d}{=} \NN(1)$ and 
for any $0 \leq j \leq m_0 - 1$, we have $\NN^\prime_j \stackrel{d}{=} \NN(1/m_0)$. 
Suppose there exists a DyBW process $\{ \WW(\tau) : \tau \in [0,1] \}$ 
define on the same probability space such that we have 
$$
\{ \pi : \pi \in \WW(\tau) \text{ for some } \tau \in [j/m_0 , (j+1)/m_0)\cap \Q\} 
\subset \NN^\prime_j .
$$
Then we must have 
$$
\bigl( \NN, \NN_0, \cdots , \NN_{m_0 - 1}  \bigr) = \bigl( \NN^\prime,
\NN^\prime_0 , \cdots , \NN^\prime_{m_0-1} \bigr ) \text{ a.s. } 
$$  
where the l.h.s. denotes the random vector of `corresponding' 
Brownian nets for the  DyBW process $\{ \WW(\tau) : \tau \in [0,1] \}$.
\end{corollary}
The proof follows from the same argument as the earlier one.
Now we are ready to prove Proposition \ref{prop:JtConv_Nets}. 

{\noindent}{\bf Proof of Proposition \ref{prop:JtConv_Nets} :} In  Section 
\ref{subsec:FDD} we proved finite dimensional convergence for the scaled DyDW 
to the DyBW. This implies following distributional convergence 
$$
\{{\cal X}_n(\tau) : \tau \in [0,1] \cap \Q\} \Rightarrow 
\{ \WW(\tau) : \tau \in [0,1] \cap \Q\} \text{ as }n \to \infty,
$$ 
as a sequence of ${\cal H}^\N$ valued random variables. 

We consider the sequence of ${\cal H}^{m_0 + 1} \times {\cal H}^\N$ valued random variables
$$
\bigl\{ \bigl ( \NN_n, \NN_{0,n}, \cdots , \NN_{m_0 -1, n} \bigr), 
\{{\cal X}_n(\tau) : \tau \in [0,1] \cap \Q\} : n \in \N  \bigr\}.
$$ 
It is not difficult to see that the above sequence is tight. 

Let $\Bigl ( \bigl({\cal Z}, {\cal Z}_0, \cdots , {\cal Z}_{m_0 -1} \bigr), 
\{ \WW(\tau) : \tau \in [0,1] \cap \Q\} \Bigr )$ be a subsequential limit of of the above sequence.
By Skorohod's embedding theorem we assume that we are working on a probability space 
where we have almost sure convergence. 
By the work of \cite{SS08} for all $0 \leq j \leq m_0 -1$ we have that 
$$
{\cal Z} \stackrel{d}{=} \NN(1) \text{ and }{\cal Z}_j \stackrel{d}{=} \NN(1/m_0).
$$
Further our construction ensures that for any $0 \leq j \leq m_0 - 1$ we have, 
$$
\pi^n \in {\cal X}_n(\tau) \text{ for some }\tau \in [j/m_0, (j+1)/m_0) \text{ only if }
\pi^n \in \NN_{j,n}.
$$
In other words, for all $0 \leq j \leq m_0 - 1$ we have 
$$
\{\pi^n : \pi^n \in {\cal X}_n(\tau) \text{ for some } \tau \in [j/m_0,(j+1)/m_0) \cap \Q \}
\subseteq \NN_{j,n}.
$$
Hence, the limiting random variable ${\cal Z}^j$ must contain 
the set of paths $\{\pi \in \WW(\tau) \text{ for some } \tau \in [j/m_0, (j+1)/m_0) \cap \Q\}$. 
By Corollary \ref{cor:Correspond_NetVector} we have that 
$$
({\cal Z}, {\cal Z}_0, \cdots , {\cal Z}_{m-1}) \stackrel{d}{=} (\NN, \NN_0, \cdots, \NN_{m_0-1}).
$$
This completes the proof.
\qed

\noindent{\bf Acknowledgement:} Part of the work on this paper was done when K.S. visited K.R at NYU Abu Dhabi. Both authors wish to thank NYU Abu Dhabi for hospitality and support.

\end{document}